\documentclass[12pt,a4paper]{article}

\usepackage[T1]{fontenc}
\usepackage[utf8]{inputenc}
\usepackage{lmodern}
\usepackage[a4paper,margin=1in]{geometry}
\usepackage{amsmath,amssymb,amsthm,mathtools}
\usepackage{enumitem}
\usepackage{microtype}
\usepackage[colorlinks=true,linkcolor=blue,citecolor=blue,urlcolor=blue]{hyperref}
\usepackage{xcolor}
\usepackage{comment}

\numberwithin{equation}{section}

\newtheorem{theorem}{Theorem}[section]
\newtheorem{proposition}[theorem]{Proposition}
\newtheorem{lemma}[theorem]{Lemma}
\newtheorem{corollary}[theorem]{Corollary}
\newtheorem{definition}[theorem]{Definition}
\newtheorem{remark}[theorem]{Remark}

\newcommand{\Torus}{\mathbb T}
\newcommand{\Dt}{\partial_t}
\newcommand{\D}{D}
\newcommand{\gam}{\gamma_d}
\newcommand{\divg}{\operatorname{div}_{\gamma_d}}

\newcommand{\ip}[2]{\left\langle #1,#2\right\rangle}

\newcommand{\R}{\mathbb R}

\newcommand{\Om}{\Omega}

\newcommand{\dd}{\,d}
\newcommand{\norm}[1]{\left\lVert #1\right\rVert}

\newcommand{\pair}[2]{\left\langle #1,#2\right\rangle}

\newcommand{\parbd}{\partial_{\mathrm{par}}\Omega_T}
\def\p{\partial}
\def \div {{\text{\rm div}}}

\newcommand{\erica}[1]{\textcolor{purple}{#1}}

\newcommand{\anna}[1]{\textcolor{blue}{#1}}
\newcommand{\an}[1]{\textcolor{green!75!black}{#1}}

\title{A variational approach to Ornstein–Uhlenbeck evolution equations and obstacle problems}
\author{Erica Ipocoana\thanks{Freie Universit\"at Berlin, Department of Mathematics and Computer Science, Arnimallee 9, 14195 Berlin, Germany. Email: \texttt{erica.ipocoana@fu-berlin.de}.}
\and
Annalaura Rebucci\thanks{Max Planck Institute for Mathematics in the Sciences, Inselstrasse 22, 04103 Leipzig, Germany. Email: \texttt{annalaura.rebucci@mis.mpg.de}.}}
\date{\today}

\begin{document}
\maketitle

\begin{abstract}
We develop a variational formulation of the Ornstein--Uhlenbeck evolution
equation based on a Gaussian divergence structure, in which the drift term
is absorbed into the first variation of a weighted Dirichlet energy rather
than treated as a lower-order perturbation. For the linear equation, the
distributional, energy weak, and variational notions of solution coincide
and determine the same unique solution. The same framework yields a
variational formulation of the associated parabolic obstacle problem,
for which we prove existence and then
uniqueness under a standard additional regularity assumption on the obstacle.
The same mechanism applies to nonlinear Gaussian gradient flows associated
with convex integrands of standard $p$-growth. Finally, we extend the
Gaussian-divergence formulation to the force-free kinetic Fokker--Planck
equation with periodic position variables, establishing the same
equivalence for the linear equation.
\end{abstract}


\section{Introduction}\label{sec:intro}

We develop a variational approach to Ornstein-Uhlenbeck evolution equations, their nonlinear analogues, and the associated parabolic obstacle problem. We also discuss how the same approach extends to the kinetic Fokker-Planck equation. Our main idea is to introduce a Gaussian formulation that incorporates the Ornstein-Uhlenbeck drift into the divergence structure and allows the equations to be treated variationally. In particular, the drift becomes part of the first variation of a weighted energy, rather than being treated separately as a lower-order term. This provides a common variational mechanism that can be used in the different settings considered below.

\subsection{The Ornstein-Uhlenbeck equation}
\label{subsec:intro-ou}\label{rem:intro-ou-specialization}

Our basic model is the Ornstein-Uhlenbeck equation 
\begin{equation}\label{O-U}
\p_t u= \Delta u - x\cdot D u,\quad \textrm{in } \Omega_T,
\end{equation}
where $\Omega_T:=\Omega \times (0,T)$ denotes the space-time cylinder, $\Omega\subset \R^d$ is a bounded Lipschitz domain and $T >0$. Throughout the paper $D$ denotes the spatial gradient in $\R^d$ and $\p_t$ the derivative with respect to time $t$, while points in space-time are denoted by $z=(x,t)$ with $x \in \R^d$ and $t\in \R$.
  
For a smooth vector field $a:\mathbb R^d\to\mathbb R^d$, we define its Gaussian divergence by
\[
 \divg a:=\operatorname{div}a-x\cdot a.
\]
Since
\[
 \divg Du=\Delta u-x\cdot Du,
\]
equation \eqref{O-U} can be written as
\begin{equation}\label{eq:intro-ou-gaussian}
\p_t u = \div_{\gamma_d}D_{\xi}f \left( Du \right), \quad \textrm{in } \Omega_T,
\end{equation}
for $f(\xi)=\frac{1}{2}|\xi|^2$, since $D_{\xi}f(Du)=Du$.
The relevance of this formulation comes from the Gaussian integration-by-parts identity
\begin{equation}\label{eq:intro-ibp}
 \int_\Omega \phi\,\divg Du\,d\gamma_d
 =-\int_\Omega Du\cdot D\phi\,d\gamma_d
\end{equation}
for test functions $\phi$ with zero trace on $\partial\Omega$. Thus $-\divg Du$ is the first variation of the Gaussian Dirichlet energy
\[
 \mathcal E(u):=\frac12\int_\Omega |Du|^2\,d\gamma_d.
\]
Consequently, \eqref{eq:intro-ou-gaussian} can be viewed as the $L^2(\Omega,\gamma_d)$-gradient flow of $\mathcal E$.

We prescribe nonhomogeneous Cauchy--Dirichlet data on the parabolic boundary
\[
 \partial_{\mathrm{par}}\Omega_T
 :=(\partial\Omega\times(0,T))\cup(\Omega\times\{0\}),
\]
and consider
\begin{equation}\label{eq:main-ou-problem}
\begin{cases}
 \partial_tu-\divg Du=0&\text{in }\Omega_T,\\
 u=g&\text{on }\partial_{\mathrm{par}}\Omega_T.
\end{cases}
\end{equation}
In the scalar case we use the Hilbert triple
\begin{equation}\label{eq:hilbert-triple}
 V:=H^1_0(\Omega,\gamma_d),\qquad
 H:=L^2(\Omega,\gamma_d),\qquad
 V^*:=H^{-1}(\Omega,\gamma_d),
\end{equation}
and write
\[
 H^1_g(\Omega,\gamma_d):=g+H^1_0(\Omega,\gamma_d).
\]
We assume
\begin{equation}\label{eq:ou-data-assumption}
 g\in L^2(0,T;H^1(\Omega,\gamma_d))\cap C([0,T];H),
 \qquad
 \partial_tg\in L^2(0,T;V^*).
\end{equation}

The classical energy weak formulation requires a time derivative of the solution in $L^2(0,T;V^*)$. The evolutionary variational formulation has a different structure: the time derivative falls on the comparison map, while no time derivative of the solution is assumed a priori. This type of formulation was developed for nonlinear parabolic equations and systems, in particular in the works of B\"ogelein--Duzaar--Marcellini \cite{BDM2013ARMA,BDM2014JDE} and Marcellini \cite{Marcellini2020}.

In the quadratic Ornstein-Uhlenbeck case, the comparison inequality takes the form
\begin{equation}\label{eq:main-variational-inequality-intro}
\begin{aligned}
 &\int_0^\tau\pair{\partial_tv}{v-u}_{V}\,dt
 +\frac12\int_0^\tau\int_\Omega
       \bigl(|Dv|^2-|Du|^2\bigr)\,d\gamma_d\,dt\\
 &\qquad\geq
 \frac12\|(v-u)(\cdot,\tau)\|_H^2
 -\frac12\|v(\cdot,0)-g(\cdot,0)\|_H^2,
\end{aligned}
\end{equation}
for every $\tau\in(0,T]$ and every comparison map
\begin{equation}\label{eq:ou-comparison-class}
 v\in L^2\bigl(0,\tau;H^1_g(\Omega,\gamma_d)\bigr),
 \qquad
 \partial_tv\in L^2(0,\tau;V^*).
\end{equation}
The unknown itself is only required to belong to
\[
 u\in L^2\bigl(0,T;H^1_g(\Omega,\gamma_d)\bigr)\cap C([0,T];H),
 \qquad u(\cdot,0)=g(\cdot,0).
\]
The general definition, which includes nonlinear Gaussian energies, is given in Subsection~\ref{subsec:intro-gaussian} below. Our first aim is to prove that, for the Ornstein--Uhlenbeck equation, the distributional, energy weak, and variational formulations determine the same solution.

\subsection{The Ornstein--Uhlenbeck obstacle problem}
\label{subsec:intro-obstacle}
In the scalar case, the Gaussian Dirichlet energy also gives a variational
formulation of the parabolic obstacle problem.  We retain the Hilbert triple
\eqref{eq:hilbert-triple} and the parabolic datum $g$ from
\eqref{eq:ou-data-assumption}.

Given an obstacle $\psi$ with $g\geq\psi$
almost everywhere, the admissible class on $\Omega_\tau:=\Omega\times(0,\tau)$ is
\[
 K_{\psi,g}(\Omega_\tau)
 :=\bigl\{w\in L^2(0,\tau;H^1_g(\Omega,\gamma_d)):
                 w\geq\psi\ \text{a.e. in }\Omega_\tau\bigr\}.
\]
Thus the obstacle restricts the comparison class, while the energy and the
Cauchy--Dirichlet conditions remain those of the Ornstein--Uhlenbeck problem.
A solution belongs to $K_{\psi,g}(\Omega_T)\cap C([0,T];H)$, has initial
value $g(\cdot,0)$, and satisfies
\begin{equation}\label{eq:intro-obstacle-variational}
\begin{aligned}
 &\int_0^\tau\pair{\partial_t v}{v-u}_{V}\,dt
   +\frac12\int_0^\tau\int_\Omega
                  \bigl(|Dv|^2-|Du|^2\bigr)\,d\gamma_d\,dt\\
 &\qquad\geq
   \frac12\|(v-u)(\cdot,\tau)\|_H^2
   -\frac12\|v(\cdot,0)-g(\cdot,0)\|_H^2
\end{aligned}
\end{equation}
for every $\tau\in(0,T]$ and every $v\in K_{\psi,g}(\Omega_\tau)$ with
$\partial_t v\in L^2(0,\tau;V^*)$.  This is the constrained version of
\eqref{eq:main-variational-inequality-intro}; the precise definition is given in
Definition~\ref{def:obstacle-solution}.  For sufficiently regular solutions,
it corresponds to
\[
 u\geq\psi,\qquad
 \partial_tu-\divg Du\geq0,\qquad
 (u-\psi)(\partial_tu-\divg Du)=0,
\]
together with the prescribed parabolic boundary data.

Section~\ref{sec:obstacle} studies this problem by following the variational
approach of B\"ogelein--Duzaar--Scheven \cite{BDS2017}.  It first treats
regular data by implicit time discretization and minimization of the
Gaussian Dirichlet energy over the obstacle constraint.  It then considers
approximation of the obstacle and of the parabolic datum under
\[
 \psi\in L^2(0,T;H^1(\Omega,\gamma_d)),\qquad
 g\in L^2(0,T;H^1(\Omega,\gamma_d))\cap H^1(0,T;H),\qquad
 g\geq\psi\ \text{a.e.}
\]
The existence argument is developed under
\eqref{eq:obstacle-sharp-hypothesis}, whereas the uniqueness results,
Theorem~\ref{thm:obstacle-uniqueness}, additionally assumes
$\partial_t\psi\in L^2(0,T;H)$.  

Beyond its intrinsic interest, the parabolic obstacle problem for
Ornstein--Uhlenbeck and Kolmogorov-type operators is also motivated by
mathematical finance, where it models the pricing of American-style and
path-dependent options with early-exercise features. The regularity
theory and the obstacle problem for the associated degenerate Kolmogorov
operators were developed by the authors in \cite{IpocoanaRebucci2021} and
\cite{AnceschiRebucci}, respectively. The present Gaussian-divergence approach
offers an alternative, energy-based route to the model
Ornstein--Uhlenbeck case.

\subsection{Nonlinear equations in Gaussian divergence form}
\label{subsec:intro-gaussian}

The Gaussian formulation is not restricted to the quadratic Ornstein--Uhlenbeck energy. Let $N\in\mathbb N$ and let
\[
 1<p<\infty,\qquad
 p\geq\frac{2d}{d+2},\qquad
 p':=\frac{p}{p-1}.
\]
We consider a Carath\'eodory integrand
\[
 f:\Omega\times\mathbb R^N\times\mathbb R^{N\times d}
 \longrightarrow[0,+\infty)
\]
satisfying the following assumptions:
\begin{enumerate}[label=\textup{(H\arabic*)},leftmargin=2.5em]
\item\label{hyp:p-car}
for almost every $x\in\Omega$, the map
$(u,\xi)\mapsto f(x,u,\xi)$ is convex and of class $C^1$ on
$\mathbb R^N\times\mathbb R^{N\times d}$;

\item\label{hyp:p-growth}
there exist constants $0<\nu\leq L<\infty$ and a nonnegative function
$a\in L^1(\Omega,\gamma_d)$ such that
\begin{equation}\label{eq:standard-p-growth}
 \nu|\xi|^p-a(x)
 \leq f(x,u,\xi)
 \leq L\bigl(1+|u|^p+|\xi|^p\bigr)
\end{equation}
for almost every $x\in\Omega$ and every
$(u,\xi)\in\mathbb R^N\times\mathbb R^{N\times d}$;

\item\label{hyp:p-derivative}
the first derivatives satisfy
\begin{equation}\label{eq:first-derivative-growth}
 |D_uf(x,u,\xi)|+|D_\xi f(x,u,\xi)|
 \leq L\bigl(1+|u|^{p-1}+|\xi|^{p-1}\bigr)
\end{equation}
for almost every $x\in\Omega$ and every
$(u,\xi)\in\mathbb R^N\times\mathbb R^{N\times d}$.
\end{enumerate}

For the Gaussian energy
\begin{equation}\label{eq:intro-general-energy}
 \mathcal F_{\gamma_d}(u)
 :=\int_\Omega f(x,u,Du)\,d\gamma_d,
\end{equation}
the corresponding negative $L^2(\Omega,\gamma_d;\mathbb R^N)$-gradient flow is
\begin{equation}\label{eq:p-main-gaussian-intro}
\begin{cases}
 \partial_tu=\divg D_\xi f(x,u,Du)-D_uf(x,u,Du)
      &\text{in }\Omega_T,\\
 u=g&\text{on }\partial_{\mathrm{par}}\Omega_T,
\end{cases}
\end{equation}
where $\divg$ acts row-wise on matrix fields.

For a spatial datum $g$ we set
\[
  W^{1,p}_g(\Omega,\gamma_d;\mathbb R^N)
  :=\bigl\{w\in W^{1,p}(\Omega,\gamma_d;\mathbb R^N):
       w-g\in W^{1,p}_0(\Omega,\gamma_d;\mathbb R^N)\bigr\}.
\]
When $g=g(\cdot,t)$, this notation is understood pointwise in time.
Thus $u=g$ on $\partial_{\mathrm{par}}\Omega_T$ means
\[
  u\in L^p\bigl(0,T;W^{1,p}_g(\Omega,\gamma_d;\mathbb R^N)\bigr),
  \qquad u(0)=g(0)\quad\text{in }L^2(\Omega,\gamma_d;\mathbb R^N).
\]

Set
\begin{equation}\label{eq:p-triple}
  V_p:=W^{1,p}_0(\Omega,\gamma_d;\mathbb R^N),
  \qquad H:=L^2(\Omega,\gamma_d;\mathbb R^N),
  \qquad V_p^*:=W^{-1,p'}(\Omega,\gamma_d;\mathbb R^N)=(V_p)^*.
\end{equation}
The condition $p\geq2d/(d+2)$ yields the dense continuous embedding
$V_p\hookrightarrow H$ and hence the evolution triple
$V_p\hookrightarrow H\equiv H^*\hookrightarrow V_p^*$.  The embedding is
compact when $p>2d/(d+2)$.
We assume that the parabolic datum satisfies
\begin{equation}\label{eq:p-data-intro}
  g\in L^p\bigl(0,T;W^{1,p}(\Omega,\gamma_d;\mathbb R^N)\bigr)
       \cap C([0,T];H),
  \qquad
  \partial_tg\in L^{p'}(0,T;V_p^*).
\end{equation}

\begin{definition}[Energy weak solution]\label{def:p-weak-intro}
A map $u$ is an energy weak solution of \eqref{eq:p-main-gaussian-intro} if
\begin{equation}\label{eq:p-weak-class}
  {u\in L^p\bigl(0,T;W^{1,p}_g(\Omega,\gamma_d;\mathbb R^N)\bigr)} \cap C([0,T];H),
  \qquad \partial_tu\in L^{p'}(0,T;V_p^*),
\end{equation}
$u(\cdot,0)=g(\cdot,0)$ in $H$, and
\begin{equation}\label{eq:p-weak}
\begin{aligned}
  \int_0^T\pair{\partial_tu}{\phi}_{V_p}\,dt
  &+\int_0^T\int_\Omega D_\xi f(x,u,Du)\cdot D\phi\,d\gamma_d\,dt\\
  &+\int_0^T\int_\Omega {D_u f(x,u,Du)}\cdot\phi\,d\gamma_d\,dt=0
\end{aligned}
\end{equation}
for every $\phi\in L^p(0,T;V_p)$.
\end{definition}

\begin{definition}[Variational solution]\label{def:p-var-intro}
A map
\[
 {u\in L^p\bigl(0,T;W^{1,p}_g(\Omega,\gamma_d;\mathbb R^N)\bigr)}
  \cap C([0,T];H),
  \qquad u(\cdot,0)=g(\cdot,0),
\]
is a variational solution of \eqref{eq:p-main-gaussian-intro} if, for every
$\tau\in(0,T]$ and every comparison map $v$ satisfying
\begin{equation}\label{eq:p-comparison-class}
  {v\in L^p\bigl(0,\tau;W^{1,p}_g(\Omega,\gamma_d;\mathbb R^N)\bigr)},
  \qquad \partial_tv\in L^{p'}(0,\tau;V_p^*),
\end{equation}
one has
\begin{equation}\label{eq:p-variational-intro}
\begin{aligned}
 &\int_0^\tau\pair{\partial_tv}{v-u}_{V_p}\,dt
 +\int_0^\tau\int_\Omega
      \bigl[f(x,v,Dv)-f(x,u,Du)\bigr]d\gamma_d\,dt\\
 &\qquad\ge
 \frac{1}{2}\norm{(v-u)(\cdot,\tau)}_H^2
 -\frac{1}{2}\norm{v(\cdot,0)-g(\cdot,0)}_H^2.
\end{aligned}
\end{equation}
\end{definition}
The term ${D_u f}$ does not appear separately in \eqref{eq:p-variational-intro}; it
is encoded in the full energy difference.  This is why convexity is required
in the pair $({u},\xi)$ rather than only in $\xi$. Taking
\[
 p=2,\qquad f(x,u,\xi)=\frac12|\xi|^2,
\]
gives exactly the quadratic comparison inequality \eqref{eq:main-variational-inequality-intro}. Thus the Ornstein--Uhlenbeck equation is the quadratic case of the general formulation, while the obstacle problem is its constrained version.

\subsection{The kinetic Fokker--Planck equation}
\label{subsec:intro-kinetic}

The same point of view also has a kinetic counterpart. Consider the force-free kinetic Fokker--Planck equation
\begin{equation}\label{eq:intro-kinetic-equation}
 (\partial_t+v\cdot D_x)u
 =\operatorname{div}_{\gamma_v}D_vu
 =\Delta_vu-v\cdot D_vu
\end{equation}
on
\[
 (0,T)\times\mathbb T^d_x\times\mathbb R^d_v.
\]
Here the appropriate reference measure is the mixed measure
\[
 dx\,d\gamma_d(v).
\]
With respect to this measure, diffusion in the velocity variable is symmetric, while the transport operator $v\cdot D_x$ is skew. Thus the Gaussian structure acts in the variable in which the Ornstein--Uhlenbeck diffusion is present.

The natural evolution derivative in this setting is
\[
 Y:=\partial_t+v\cdot D_x.
\]
At the energy level, $Yu$ can be controlled in a suitable dual space even though $\partial_tu$ and $v\cdot D_xu$ need not be controlled separately. This leads to a kinetic comparison formulation in which $Y$ replaces $\partial_t$ and $D_v$ replaces the full spatial gradient. In Section~\ref{sec:kinetic} we show, using the kinetic Cauchy theory of \cite{AAMN}, that the resulting variational formulation is equivalent to the energy weak formulation for the linear equation. We also record the corresponding formal nonlinear formula. The kinetic obstacle problem requires additional arguments and is left for future work.

\subsection{Main results}\label{subsec:intro-main-results}

We now state the main results for the problems introduced above.

\begin{theorem}[Convex Gaussian gradient flows with standard $p$-growth]
\label{thm:intro-general-p}
Assume \ref{hyp:p-car}--\ref{hyp:p-derivative} and \eqref{eq:p-data-intro}.  Then
\eqref{eq:p-main-gaussian-intro} has a unique energy weak solution in the sense of
Definition~\ref{def:p-weak-intro} and a unique variational solution in the sense of
Definition~\ref{def:p-var-intro}, and these solutions coincide.  Consequently,
the energy weak and variational formulations are equivalent.
\end{theorem}

The Ornstein--Uhlenbeck equation is the quadratic case of this result. In that case the distributional formulation can also be included in the equivalence.

\begin{corollary}[Ornstein--Uhlenbeck equation]\label{cor:intro-main-ou}
Assume \eqref{eq:ou-data-assumption}.  Then the Ornstein--Uhlenbeck problem
\eqref{eq:main-ou-problem} has a unique variational solution and a unique
energy weak solution, and the two coincide.  Moreover, every distributional
weak solution in the spatial energy class is automatically an energy weak
solution.  Thus, in the Ornstein--Uhlenbeck case, the distributional, energy
weak, and variational formulations are equivalent in their respective
natural classes.
\end{corollary}

Although Corollary~\ref{cor:intro-main-ou} is formally contained in the
general parabolic theory as far as the energy weak and variational notions
are concerned, we prove the Ornstein--Uhlenbeck case first in
Section~\ref{sec:ou-operator}.  The quadratic setting isolates the two main
ideas---the weak-to-variational argument and the regularized-midpoint
uniqueness argument---without the additional notation of the nonlinear
$p$-growth theory.  Section~\ref{sec:general-operators} then returns to the
full statement of Theorem~\ref{thm:intro-general-p}.

For the obstacle problem, the regularity of the obstacle determines whether one can require a continuous $H$-valued representative.

\begin{theorem}[Ornstein--Uhlenbeck obstacle problem]\label{thm:intro-obstacle}
Assume \eqref{eq:obstacle-sharp-hypothesis}. Then there exists a variational
solution of the obstacle problem in the sense of
Definition~\ref{def:obstacle-solution}. If, in addition,
$\partial_t\psi\in L^2(0,T;H)$, this solution has a representative in
$C([0,T];H)$, the variational inequality holds for every
$\tau\in(0,T]$, and the solution is unique.
\end{theorem}

Finally, in the kinetic setting we obtain the corresponding equivalence for the linear equation.

\begin{proposition}[Linear kinetic Fokker--Planck equation]\label{prop:intro-kinetic}
For every
\[
 u_0\in L^2(\mathbb T^d\times\mathbb R^d,dx\,d\gamma_d(v)),
\]
the energy weak solution of \eqref{eq:intro-kinetic-equation} is the unique
variational solution in the kinetic comparison class introduced in
Section~\ref{sec:kinetic}. In particular, the weak and variational
formulations coincide.
\end{proposition}

\subsection{Outline of the paper}\label{subsec:intro-outline}

Section~\ref{sec:preliminaries} collects the Gaussian Sobolev framework, the
weighted divergence and integration-by-parts formula, the relevant evolution
triples, and the time regularization used in the uniqueness argument. Section~\ref{sec:ou-operator} contains the proof of the equivalence of distributional, energy weak, and variational solutions for the Ornstein-Uhlenbeck equation. Section~\ref{sec:obstacle} is devoted to the corresponding parabolic obstacle problem. Section~\ref{sec:general-operators} treats nonlinear equations in Gaussian divergence form associated with convex energies of standard $p$-growth and discusses representative examples. Finally, Section~\ref{sec:kinetic} presents the kinetic Fokker-Planck equation in mixed Gaussian variables.

\section{Preliminaries and notation}\label{sec:preliminaries}

In this section, we first introduce the Gaussian Sobolev spaces and integration-by-parts
 identities, and then explain the time traces
 and regularization needed for the comparison formulation. We also record
 the Ornstein--Uhlenbeck identities used in Section~\ref{sec:ou-operator}.
 The kinetic spaces are introduced separately in
 Section~\ref{sec:kinetic}.


We recall that $d\ge 1$, $T>0$, $\Omega_T:=\Omega\times(0,T)$, and
$D$ denotes the weak gradient with respect to the spatial variable $x\in\mathbb R^d$. Moreover, we write
\[
  d\gamma_d(x)=\rho_d(x)\,dx,
  \qquad
  \rho_d(x):=(2\pi)^{-d/2}e^{-|x|^2/2},
\]
for the standard Gaussian measure on $\mathbb R^d$. For matrix fields, both $\cdot$ and $:$
 denote the Frobenius product. The notation
 $\langle\ell,\phi\rangle_{V_p}$ denotes the pairing of
 $\ell\in V_p^*$ with $\phi\in V_p$, and similarly for $V$ and $Z$.

\subsection{Gaussian Sobolev spaces}\label{subsec:weighted-spaces}
\subsection{The whole space}
Following \anna{\cite{LunardiMetafunePallara2020}}, we recall the definition of Gaussian Sobolev spaces in the finite-dimensional setting. For $1\le p<\infty$, the Lebesgue spaces with respect to the Gaussian measure are defined as follows
\[
  L^p(\mathbb R^d,\gamma_d)
  :=\left\{u:\mathbb R^d\to\mathbb R\text{ measurable}:
  \int_{\mathbb R^d}|u|^p\,d\gamma_d<\infty\right\}.
\]
Let $C_b^1(\mathbb R^d)$ be the space of bounded $C^1$ functions with bounded first
order derivatives. The gradient operator
\[
  D:C_b^1(\mathbb R^d)\subset L^p(\mathbb R^d,\gamma_d)
  \to L^p(\mathbb R^d,\gamma_d;\mathbb R^d)
\]
is closable. Its closure has domain
$W^{1,p}(\mathbb R^d,\gamma_d)\subset L^p(\mathbb R^d,\gamma_d)$
and takes values in $L^p(\mathbb R^d,\gamma_d;\mathbb R^d)$.
Equivalently, $u\in W^{1,p}(\mathbb R^d,\gamma_d)$ if there are functions $u_j\in C_b^1(\mathbb R^d)$ such that
\[
  u_j\to u \quad\text{in }L^p(\mathbb R^d,\gamma_d),
  \qquad
  D u_j\to G \quad\text{in }L^p(\mathbb R^d,\gamma_d;\mathbb R^d),
\]
and in this case $Du:=G$. We recall that $W^{1,p}(\mathbb R^d,\gamma_d)$ is a Banach space with the graph norm
\begin{align*}
  \|u\|_{W^{1,p}(\mathbb R^d,\gamma_d)}
  :&=\|u\|_{L^p(\mathbb R^d,\gamma_d)}
  +\|Du\|_{L^p(\mathbb R^d,\gamma_d;\mathbb R^d)}\\
  &=\left( \int_{\R^d} |u|^p d \gamma_d\right)^{1/p}+ \left( \int_{\R^d}|D u|^p d \gamma_d\right)^{1/p}.
\end{align*}
For $p=2$ this is a Hilbert space; we write
$H^1(\mathbb R^d,\gamma_d):=W^{1,2}(\mathbb R^d,\gamma_d)$. All definitions extend componentwise to vector-valued maps.

For $u,\phi\in C_b^1(\mathbb R^d)$ and $i=1,\ldots,d$, the Gaussian integration-by-parts formula reads
\[
  \int_{\mathbb R^d} u\,\frac{\partial\phi}{\partial x_i}\,d\gamma_d
  =-\int_{\mathbb R^d}\phi\,\frac{\partial u}{\partial x_i}\,d\gamma_d
  +\int_{\mathbb R^d}x_i u\phi\,d\gamma_d .
\]
The weak derivatives admit the following equivalent characterization. For $u\in L^p(\mathbb R^d,\gamma_d)$ a function
$F_i\in L^p(\mathbb R^d,\gamma_d)$ is the $i$-th derivative of $u$ if
\[
  \int_{\mathbb R^d} u\,\frac{\partial \phi}{\partial x_i}\,d\gamma_d
  =-\int_{\mathbb R^d}F_i\phi\,d\gamma_d
  +\int_{\mathbb R^d}x_i u\phi\,d\gamma_d,
  \qquad \phi\in C_c^\infty(\mathbb R^d).
\]
In this case we write $D_i u:=F_i$. Equivalently,
\[
  W^{1,p}(\mathbb R^d,\gamma_d)
  =\left\{u\in W^{1,p}_{\rm loc}(\mathbb R^d):
  u,D_1u,\ldots,D_du\in L^p(\mathbb R^d,\gamma_d)\right\}.
\]

 \subsection{Bounded domains}
On the bounded Lipschitz domain $\Omega$, the density
$\rho_d$ is bounded above and bounded away from zero. Hence there are
constants $0<m_\Omega\leq M_\Omega<\infty$ such that
\begin{equation}\label{eq:norm-equivalence}
     m_\Omega\int_\Omega |u|^p\,dx
  \le \int_\Omega |u|^p\,d\gamma_d
  \le M_\Omega\int_\Omega |u|^p\,dx .
\end{equation}
 
We therefore set
\[
  W^{1,p}(\Omega,\gamma_d)
  :=\left\{u\in L^p(\Omega,\gamma_d):D_i u\in L^p(\Omega,\gamma_d),
  \ i=1,\ldots,d\right\},
\]
where $D_i u$ denotes the usual distributional derivative in $\Omega$, and equip this space with its Gaussian-weighted norm.
We also set
\[
  W^{1,p}_0(\Omega,\gamma_d):=\overline{C_c^\infty(\Omega)}^{\,W^{1,p}(\Omega,\gamma_d)},
  \qquad
  H^1_0(\Omega,\gamma_d):=W^{1,2}_0(\Omega,\gamma_d).
\]
For a datum $g$, we use the affine spaces
 $W^{1,p}_g(\Omega,\gamma_d;\mathbb R^N)$ introduced in
 Section~\ref{subsec:intro-gaussian}. In particular,
 $H^1_g:=W^{1,2}_g$ and $H^1_0:=W^{1,2}_0$.
For $1<p<\infty$ and $p'=p/(p-1)$, we write
\[
  W^{-1,p'}(\Omega,\gamma_d;\mathbb R^N)
  :=\bigl(W^{1,p}_0(\Omega,\gamma_d;\mathbb R^N)\bigr)^*.
\]

\begin{proposition}\label{propembedding}
Let $\Omega\subset\mathbb R^d$ be bounded and Lipschitz and let
$1\le p<\infty$. Then:

\begin{enumerate}[label=\textup{(\roman*)}]
\item the usual trace operator is continuous on
$W^{1,p}(\Omega,\gamma_d)$, and
$W^{1,p}_0(\Omega,\gamma_d)$ is the subspace of functions with zero trace;

\item the embedding
\[
W^{1,p}(\Omega,\gamma_d)
\hookrightarrow L^p(\Omega,\gamma_d)
\]
is compact;

\item the weighted Poincar\'e inequality holds: there is a constant $C_{p,\Omega}>0$ such that
\begin{equation}\label{eq:weighted-poincare}
 \norm{u}_{L^p(\Omega,\gamma_d)}\le C_{p,\Omega}\norm{Du}_{L^p(\Omega,\gamma_d)}
 \qquad\text{for every }u\in W^{1,p}_0(\Omega,\gamma_d).
\end{equation}

\item if $1<p<\infty$ and
\[
p\ge \frac{2d}{d+2},
\]
then
\[
W^{1,p}_0(\Omega,\gamma_d;\mathbb R^N)
\hookrightarrow
L^2(\Omega,\gamma_d;\mathbb R^N)
\]
continuously and densely. If
\[
p>\frac{2d}{d+2},
\]
the embedding is compact.
\end{enumerate}
\end{proposition}

\subsection{Gaussian divergence and the Dirichlet realization}
The Gaussian divergence is the negative Hilbert-space
 adjoint of the closed gradient. Let
\[
  D:H^1(\mathbb R^d,\gamma_d)\subset L^2(\mathbb R^d,\gamma_d)
  \longrightarrow L^2(\mathbb R^d,\gamma_d;\mathbb R^d),
  \qquad Tu=Du,
\]
be the closed gradient operator. A vector field
$a\in L^2(\mathbb R^d,\gamma_d;\mathbb R^d)$ belongs to the domain of the
Gaussian divergence if there exists $h\in L^2(\mathbb R^d,\gamma_d)$ such
that
\begin{equation}\label{eq:gaussian-divergence-adjoint}
  \int_{\mathbb R^d}Du\cdot a\,d\gamma_d
  =-\int_{\mathbb R^d}u h\,d\gamma_d
  \qquad\text{for every }u\in H^1(\mathbb R^d,\gamma_d).
\end{equation}
In this case we set $\divg a:=h$. Equivalently,
$\divg=-D^*$ on its domain. If $a\in C_c^1(\mathbb R^d;\mathbb R^d)$,
then 
\begin{equation}\label{eq:pointwise-gaussian-divergence}
  \divg a=\operatorname{div}a-x\cdot a
  =\sum_{i=1}^d(\partial_i a_i-x_i a_i).
\end{equation}
In particular, for smooth $u$ for which the expression is defined,
\[
  \divg Du=\Delta u-x\cdot Du,
\]
which is the Ornstein--Uhlenbeck operator in Gaussian divergence form.

For the weak formulations on the bounded domain $\Omega$, we use the
corresponding zero-trace dual extension. Let $1<p<\infty$ and
$a\in L^{p'}(\Omega,\gamma_d;\mathbb R^d)$. We define
$\divg a\in W^{-1,p'}(\Omega,\gamma_d)$ by
\begin{equation}\label{eq:gaussian-ibp-zero-trace}
  \langle\divg a,\phi\rangle_{W^{1,p}_0}
  :=-\int_\Omega a\cdot D\phi\,d\gamma_d,
  \qquad \phi\in W^{1,p}_0(\Omega,\gamma_d).
\end{equation}
This is well defined by H\"older's inequality. If $a$ is smooth, then for
$\phi\in C_c^\infty(\Omega)$ the usual integration-by-parts formula gives
\[
  -\int_\Omega a\cdot D\phi\,d\gamma_d
  =\int_\Omega(\operatorname{div}a-x\cdot a)\phi\,d\gamma_d,
\]
so the dual definition agrees with the pointwise expression
\eqref{eq:pointwise-gaussian-divergence}.

For a matrix field $A=(A_i^\alpha)_{\alpha=1,\ldots,N}^{i=1,\ldots,d}$,
we use the row-wise convention
\[
  \bigl(\divg A\bigr)^\alpha
  :=\sum_{i=1}^d\bigl(\partial_iA_i^\alpha-x_iA_i^\alpha\bigr)
\]
whenever $A$ is smooth. More generally, if
$A\in L^{p'}(\Omega,\gamma_d;\mathbb R^{N\times d})$, we define
\[
  \divg A\in W^{-1,p'}(\Omega,\gamma_d;\mathbb R^N)
\]
by
\begin{equation}\label{eq:matrix-gaussian-divergence-dual}
  \langle\divg A,\phi\rangle
  :=-\int_\Omega A:D\phi\,d\gamma_d,
  \qquad
  \phi\in W^{1,p}_0(\Omega,\gamma_d;\mathbb R^N).
\end{equation}

For $p=2$ we use the Hilbert triple \eqref{eq:hilbert-triple},
 with componentwise spaces for vector-valued maps. The Gaussian Dirichlet form is
\[
  \mathfrak a_{\gamma_d}(u,\phi)
  :=\int_\Omega Du:D\phi\,d\gamma_d,
  \qquad u,\phi\in V.
\]
By Proposition~\ref{propembedding}, this form is continuous and coercive on
$V$. We denote by $A_{\gamma_d}:V\to V^*$ the associated operator, namely
\begin{equation}\label{eq:quadratic-dirichlet-operator}
  \pair{A_{\gamma_d}u}{\phi}_{V}
  :=\int_\Omega Du:D\phi\,d\gamma_d,
  \qquad u,\phi\in V.
\end{equation}
Thus $A_{\gamma_d}u=-\divg Du$ in $V^*$. We refer to
$A_{\gamma_d}$ as the variational Dirichlet realization of
$-\divg D$.

\subsection{Evolution triples}
Let $V_p$, $H$, and $V_p^*$ be as in \eqref{eq:p-triple}.
 For $1<p<\infty$ with $p\geq2d/(d+2)$,
 Proposition~\ref{propembedding} gives
\[
  V_p\hookrightarrow H\equiv H^*\hookrightarrow V_p^*.
\]
For $p=2$ this is the Hilbert triple \eqref{eq:hilbert-triple}.

\begin{proposition}[Evolution-space continuity and energy identity]
{\color{black}Let $1<p<\infty$ with $p\geq2d/(d+2)$.}  If
\[
  w\in L^p(0,T;V_p),
  \qquad
  \partial_tw\in L^{p'}(0,T;V_p^*),
\]
then $w$ has a unique representative in $C([0,T];H)$ and
\begin{equation}\label{eq:lions-magenes}
  \int_s^t\pair{\partial_rw}{w}_{V_p}\,dr
  =\frac12\|w(\cdot,t)\|_H^2-\frac12\|w(\cdot,s)\|_H^2
\end{equation}
for every $0\leq s\leq t\leq T$.
\end{proposition}
This is the standard Lions-Magenes lemma for an evolution triple, see
\cite{Lions1969,brezis}.

For the affine comparison class
 \eqref{eq:p-comparison-class}, apply this proposition to $v-g$.
 Indeed, $v-g\in L^p(0,\tau;V_p)$ and
 $\partial_t(v-g)\in L^{p'}(0,\tau;V_p^*)$. Thus $v-g$ is continuous
 into $H$, and so is $v$, because $g\in C([0,\tau];H)$.
 This justifies the endpoint values in the variational inequality without
 requiring $g$ or $v$ to have zero lateral trace. The difference $v-u$
 does have zero lateral trace, which is why the pairing with
 $\partial_t v\in V_p^*$ is well defined.

\subsection{The Ornstein--Uhlenbeck weak formulations}\label{subsec:ou-formulations}

For later use, we record explicitly the distributional and energy weak
formulations of \eqref{eq:main-ou-problem}. Under
\eqref{eq:ou-data-assumption}, a function
\[
  u\in L^2(0,T;H^1_g(\Omega,\gamma_d))\cap C([0,T];H),
  \qquad u(\cdot,0)=g(\cdot,0),
\]
is a distributional weak solution if
\begin{equation}\label{eq:main-weak-formulation}
  \int_{\Omega_T}u\,\partial_t\varphi\,d\gamma_d\,dt
  -\int_{\Omega_T}Du\cdot D\varphi\,d\gamma_d\,dt=0
\end{equation}
for every $\varphi\in C_c^\infty(\Omega_T)$.

If, in addition,
\[
  \partial_tu\in L^2(0,T;V^*),
\]
the corresponding energy weak formulation is
\begin{equation}\label{eq:main-weak-energy-fixed}
  \int_0^T\pair{\partial_tu}{\phi}_{V}\,dt
  +\int_0^T\int_\Omega Du\cdot D\phi\,d\gamma_d\,dt=0
\end{equation}
for every $\phi\in L^2(0,T;V)$. Proposition~\ref{prop:distributional-energy-ou-fixed}
shows that, in the spatial energy class above, the distributional formulation
automatically gives the additional time regularity and hence the two
formulations are equivalent.

\subsection{Time regularization}\label{subsec:time-regularization}


Definition \ref{def:p-var-intro} does not assume that a solution has a time derivative and therefore a variational solution cannot
 immediately be used as a comparison map. To overcome this difficulty, we shall use the causal exponential regularization of B\"ogelein-Duzaar-Marcellini in \cite {BDM2013ARMA} and
record only the properties needed in the uniqueness arguments below; see
\cite[Section~5.1 and Appendix~B]{BDM2013ARMA}.


\begin{lemma}[Exponential time regularization]\label{lem:time-regularization}
Let $H$ be a Hilbert space and $Z$ a reflexive Banach space
 with dense continuous embeddings
 $Z\hookrightarrow H\equiv H^*\hookrightarrow Z^*$.
 Let $1<r<\infty$, $r':=r/(r-1)$, and $\tau\in(0,T]$. Suppose that
\[
\begin{gathered}
  g\in C([0,\tau];H),
  \qquad z:=w-g\in L^r(0,\tau;Z)\cap C([0,\tau];H),\\
  z(0)=0,
  \qquad \partial_tg\in L^{r'}(0,\tau;Z^*).
\end{gathered}
\]
For $h>0$, define
\begin{equation}\label{eq:bdm-zero-seed-regularization}
  R_hz(t):=\frac1h\int_0^t e^{(s-t)/h}z(s)\,ds,
  \qquad
  w_h:=g+R_hz.
\end{equation}
Then
\[
  w_h-g\in L^r(0,\tau;Z),
  \qquad w_h\in C([0,\tau];H),
  \qquad w_h(0)=g(0),
  \qquad \partial_tw_h\in L^{r'}(0,\tau;Z^*),
\]
and
\begin{equation}\label{eq:bdm-regularization-derivative}
  \partial_tR_hz=-\frac1h(R_hz-z).
\end{equation}
Moreover,
\[
  w_h\to w\quad\text{in }L^r(0,\tau;Z)\cap C([0,\tau];H),
\]
and
\begin{equation}\label{eq:bdm-affine-one-sided}
  \limsup_{h\downarrow0}
  \int_0^\tau\pair{\partial_tw_h}{w_h-w}_{Z}\,dt\leq0.
\end{equation}
\end{lemma}

\begin{proof}
The kernel $h^{-1}e^{(s-t)/h}\mathbf 1_{0<s<t}$ is a one-sided approximate
identity, and therefore $R_hz\to z$ in $L^r(0,\tau;Z)$.  Differentiating the
integral gives \eqref{eq:bdm-regularization-derivative}.  Since
$z\in C([0,\tau];H)$, the difference $R_hz-z$ belongs to
$L^\infty(0,\tau;H)$ for every fixed $h$; hence
\[
  \partial_tw_h
  =\partial_tg-\frac1h(R_hz-z)
  \in L^{r'}(0,\tau;Z^*).
\]
The identity
\[
  R_hz(t)-z(t)
  =\frac1h\int_0^t e^{(s-t)/h}\bigl(z(s)-z(t)\bigr)\,ds-e^{-t/h}z(t)
\]
and the uniform continuity of $z$ in $H$, together with $z(0)=0$, yield
$R_hz\to z$ in $C([0,\tau];H)$.  Finally,
\[
\begin{aligned}
  \int_0^\tau\pair{\partial_tw_h}{w_h-w}\,dt
  &=-\frac1h\int_0^\tau\|w_h-w\|_H^2\,dt\\
  &\quad+\int_0^\tau\pair{\partial_tg}{w_h-w}_{Z}\,dt.
\end{aligned}
\]
The first term is nonpositive and the second tends to zero by the strong
convergence in $L^r(0,\tau;Z)$, proving
\eqref{eq:bdm-affine-one-sided}.
\end{proof}

\section{The Ornstein--Uhlenbeck operator}\label{sec:ou-operator}

This section is devoted to the proof of Corollary~\ref{cor:intro-main-ou}, i.e. the existence and uniqueness result for the solutions the Ornstein-Uhlenbeck equation \eqref{O-U}. Hence, throughout this section, we assume that the energy density is quadratic, namely $f(Du)=\frac{1}{2}|D u|^2$ to recover \eqref{O-U}. Moreover, we keep the Hilbert triple notation \eqref{eq:hilbert-triple} and assume \eqref{eq:ou-data-assumption}.

\subsection{Distributional and energy weak formulations}

\begin{proposition}[Distributional weak solutions are energy weak solutions]\label{prop:distributional-energy-ou-fixed}
Let
\[
  u\in L^2(0,T;H^1_g(\Omega,\gamma_d))\cap C([0,T];H),
  \qquad u(\cdot,0)=g(\cdot,0),
\]
satisfy the distributional identity \eqref{eq:main-weak-formulation}.  Then
$\partial_tu\in L^2(0,T;V^*)$ and $u$ satisfies the energy weak formulation
\eqref{eq:main-weak-energy-fixed}.  Conversely, every energy weak solution
satisfies \eqref{eq:main-weak-formulation}.
\end{proposition}

\begin{proof}
Since $u-g\in L^2(0,T;V)$ and
$g\in L^2(0,T;H^1(\Omega,\gamma_d))$, we have
$Du\in L^2(\Omega_T,\gamma_d\otimes dt)$.  Define
$F\in L^2(0,T;V^*)$ by
\[
  \pair{F(t)}{\phi}_{V}
  :=-\int_\Omega Du(t)\cdot D\phi\,d\gamma_d,
  \qquad \phi\in V.
\]
Testing \eqref{eq:main-weak-formulation} with functions of the form
$\eta(t)\phi(x)$, first for $\phi\in C_c^\infty(\Omega)$ and then
using density in $V$, shows that $\partial_tu=F$ in
$\mathcal D'(0,T;V^*)$.  Thus $\partial_tu\in L^2(0,T;V^*)$, and
\eqref{eq:main-weak-energy-fixed} follows by density in
$L^2(0,T;V)$.  The converse is obtained by testing the energy identity with
compactly supported smooth functions and integrating by parts in time.
\end{proof}

\subsection{Existence and the weak-to-variational implication}

{We here recall the following result contained in \cite[Theorem 10.9]{brezis} concerning the existence and uniqueness of a weak solution to our problem \eqref{eq:main-ou-problem}.}
\begin{proposition}[Existence and uniqueness of weak solutions]\label{prop:weak-existence-nonhomogeneous}
There exists a unique weak solution to the the Cauchy-Dirichlet problem \eqref{eq:main-ou-problem}. More precisely,
\[
 u\in  L^2 \left(0,T;H_g^1(\Omega,\gamma_d)\right) \cap C \left([0,T];H\right),
\]
with $u(\cdot,0)=g(\cdot,0)$ and \eqref{eq:main-weak-energy-fixed} holds. 
\end{proposition}

\begin{proof}
Set $w:=u-g$. Then \eqref{eq:main-weak-formulation} is equivalent to
\begin{equation}\label{eq:w-equation-ou}
 \pair{\partial_t w}{\phi}_{V}+\int_\Omega Dw\cdot D\phi\,d\gamma_d
 =-\pair{\partial_t g}{\phi}_{V}-\int_\Omega Dg\cdot D\phi\,d\gamma_d
\end{equation}
for every $\phi\in V$, with $w(\cdot,0)=0$ in $H$. The right-hand side of \eqref{eq:w-equation-ou} belongs to $L^2(0,T;V^*)$. The bilinear form
\[
 a(w,\phi):=\int_\Omega Dw\cdot D\phi\,d\gamma_d
\]
is continuous on $V\times V$ and coercive on $V$ by the weighted Poincar\'e inequality \eqref{eq:weighted-poincare}. The standard Lions theorem for linear parabolic equations in the triple $V\hookrightarrow H\hookrightarrow V^*$ therefore gives a unique
\[
 w\in L^2(0,T;V)\cap C([0,T];H),\qquad \partial_t w\in L^2(0,T;V^*).
\]
Then $u=w+g$ is the desired weak solution. Uniqueness follows by applying the same equation to the difference of two solutions and testing by the difference.
\end{proof}

\begin{proposition}[Weak solutions are variational solutions]\label{prop:weak-implies-var}
Every weak solution of \eqref{eq:main-ou-problem} satisfies the variational inequality \eqref{eq:main-variational-inequality-intro}.
\end{proposition}
\begin{proof}
We fix $\tau \in (0,T]$ and we consider $u$ weak solution to the Cauchy-Dirichlet problem \eqref{eq:main-ou-problem}, then for any test function $\phi \in L^2 (0,\tau; V)$ we have
\begin{eqnarray}\label{var-for}
\int_0^\tau \langle \p_t u, \phi \rangle_{V}\,dt-\int_0^\tau\int_{\Omega}\div_{\gamma_d}\left( Du\right)\,\phi\, d\gamma_d \ dt=0.
\end{eqnarray}
By integrating by parts \eqref{var-for} using the integration formula \eqref{eq:gaussian-ibp-zero-trace} with $a=Du$, we obtain
\begin{eqnarray}\label{var-for-2}
\int_0^\tau \langle \p_t u, \phi \rangle_{V}\,dt+\int_0^\tau\int_{\Omega}Du \cdot D\phi \; d\gamma_d \ dt=0,
\end{eqnarray}
for every $\phi \in L^2 (0,\tau; V)$. As $f(Du)=\frac{1}{2}|Du|^2$ is trivially convex, we exploit inequality $\frac{1}{2}|D\phi+Du|^2-\frac{1}{2}|Du|^2 \geq  Du\cdot D\phi $ in \eqref{var-for-2} and infer 
\begin{eqnarray}\label{ineq-var}
\int_0^\tau \langle \p_t u, \phi \rangle_{V}\,dt+\int_0^\tau\int_{\Omega}\left(\frac{1}{2}|D(\phi+u)|^2-\frac{1}{2}|Du|^2 \right)\, d\gamma_d \ dt\geq 0,
\end{eqnarray}
for every $\phi \in L^2 (0,\tau; V)$. We now perform the substitution $v:= u+\phi$, which clearly satisfies $v \in   u +L^2 (0,\tau; V) \subset L^2 (0,\tau; H^1(\Omega,\gamma_d))$ by Proposition \ref{propembedding} and {$u=v=g$} on $\p \Omega \times (0,\tau)$. Moreover, we have $\p_t v \in L^2 \left( 0,T;V^*\right)$, provided that we additionally assume $\p_t \phi \in L^2 \left( 0,T;V^*\right)$. We now rewrite inequality \eqref{ineq-var} in terms of the function $v=u+\phi$ as follows
\begin{align*}
&\int_0^\tau \langle \p_t v, v-u \rangle_{H^1(\Omega,\gamma_d)}\,dt+\int_0^\tau\int_{\Omega}\left(\frac{1}{2}|Dv|^2-\frac{1}{2}|Du|^2 \right)\, d\gamma_d \ dt\\
&\quad\geq \int_0^\tau \langle \p_t (v-u), v-u \rangle_{H^1(\Omega,\gamma_d)}\,dt.
\end{align*}
Thus, the previous inequality holds true for every $v \in  L^2 (0,\tau; V)$ with $\p_t v \in L^2 \left( 0,T;V^*\right)$. As a consequence of Aubin-Lions Lemma, according to Proposition \ref{propembedding}, we have $v \in C^0([0,\tau],H)$. Hence, we can apply the Fundamental Theorem of Calculus to the right-hand side of the previous inequality and get
\begin{equation*}
\begin{split}
\int_0^\tau \langle \p_t (v-u), v-u \rangle_{H^1(\Omega,\gamma_d)}\,dt =\frac{1}{2}\Vert (v-u)(\cdot,\tau)\Vert^2_{H}-\frac{1}{2}\Vert v(\cdot,0)-g(\cdot,0)\Vert^2_{H}.
\end{split}
\end{equation*}
Therefore, we obtain that inequality
\begin{align*}
&\int_0^\tau \langle \p_t v, v-u \rangle_{H^1(\Omega,\gamma_d)}\,dt+\int_0^\tau\int_{\Omega}\left(\frac{1}{2}|Dv|^2-\frac{1}{2}|Du|^2 \right)\, d\gamma_d \ dt\\
&\quad\geq \frac{1}{2}\Vert (v-u)(\cdot,\tau)\Vert^2_H-\frac{1}{2}\Vert v(\cdot,0)-g(\cdot,0)\Vert^2_H
\end{align*}
holds true for any comparison function $v \in  L^2 (0,\tau; V)$ whose time derivative satisfies $\p_t v \in L^2 \left( 0,T;V^*\right)$. This proves that $u$ is a variational solution to \eqref{eq:main-ou-problem} in the sense of Definition \ref{def:p-var-intro}.
   \end{proof}

\subsection{Uniqueness in the variational class}

\begin{proposition}[Uniqueness of variational solutions]\label{prop:var-uniqueness-ou}
The variational solution 
\begin{equation*}
u \in L^2 \left(0,T;H_g^1(\Omega,\gamma_d)\right) \cap C \left([0,T];H\right) 
\end{equation*}
to the Cauchy-Dirichlet problem \eqref{eq:main-ou-problem} is unique.
\end{proposition}

\begin{proof}
Let $u_1$ and $u_2$ be two variational solutions and fix $\tau\in(0,T]$.  Set
\[
  w:=\frac{u_1+u_2}{2}.
\]
Then $w-g\in L^2(0,\tau;V)$, $w\in C([0,\tau];H)$, and $w(\cdot,0)=g(\cdot,0)$.  Since $w$ need not have a time derivative, we use the time regularization from Lemma~\ref{lem:time-regularization} with $Z=V$ and $r=2$:
\[
  w_h:=g+R_h(w-g).
\]
Then $w_h$ is admissible, $w_h(\cdot,0)=g(\cdot,0)$,
\[
  w_h\to w\quad\text{in }L^2(0,\tau;V)\text{ and in }C([0,\tau];H),
\]
and
\begin{equation}\label{eq:ou-time-reg-limsup-fixed}
  \limsup_{h\downarrow0}\int_0^\tau
  \pair{\partial_tw_h}{w_h-w}_{V}\,dt\le0.
\end{equation}
Using $v=w_h$ as comparison map in the variational inequalities for $u_1$ and $u_2$, and adding them, gives
\begin{equation}\label{eq:ou-added-var-fixed}
\begin{aligned}
&2\int_0^\tau\pair{\partial_tw_h}{w_h-w}_{V}\,dt
 +2\int_0^\tau\int_\Omega f(Dw_h)\,d\gamma_d\,dt
 -\sum_{i=1}^2\int_0^\tau\int_\Omega f(Du_i)\,d\gamma_d\,dt  \\
&\hspace{3cm}\ge
 \frac12\sum_{i=1}^2\norm{(w_h-u_i)(\cdot,\tau)}_H^2 .
\end{aligned}
\end{equation}
Letting $h\downarrow0$ in \eqref{eq:ou-added-var-fixed}, using the strong convergence of $w_h$ and \eqref{eq:ou-time-reg-limsup-fixed}, yields
\begin{equation}\label{eq:ou-midpoint-fixed}
  2\int_0^\tau\int_\Omega f(Dw)\,d\gamma_d\,dt
  -\sum_{i=1}^2\int_0^\tau\int_\Omega f(Du_i)\,d\gamma_d\,dt
  \ge
  \frac12\sum_{i=1}^2\norm{(w-u_i)(\cdot,\tau)}_H^2 .
\end{equation}
On the other hand, convexity gives
\[
  2f(Dw)=2f\left(\frac{Du_1+Du_2}{2}\right)
  \le f(Du_1)+f(Du_2)
\]
almost everywhere.  Thus the left-hand side of \eqref{eq:ou-midpoint-fixed} is nonpositive, whereas the right-hand side is nonnegative.  Hence both vanish and
\[
  u_1(\cdot,\tau)=u_2(\cdot,\tau)\qquad\text{in }H.
\]
Since $\tau\in(0,T]$ was arbitrary and both functions are continuous into $H$, $u_1=u_2$ in $C([0,T];H)$ and hence a.e. in $\Omega_T$.
\end{proof}

\begin{proof}[Proof of Corollary~\ref{cor:intro-main-ou}]
Let $u$ be a variational solution to \eqref{eq:main-ou-problem}. First, we observe that by Proposition \ref{prop:weak-existence-nonhomogeneous} there exists $\tilde{u}$ weak solution to \eqref{eq:main-ou-problem}. By Proposition~\ref{prop:weak-implies-var}, this weak solution is a variational solution. Moreover, thanks to Proposition \ref{prop:var-uniqueness-ou}, the variational solution is unique, i.e. $\tilde{u}=u$. Hence, $u$ is also a weak solution to system \eqref{eq:main-ou-problem} and the proof of Corollary \ref{cor:intro-main-ou} is concluded.

\begin{remark}[The role of the Gaussian formulation]
The Ornstein--Uhlenbeck equation can of course be studied in the classical Lebesgue framework, where the drift term $-x\cdot Du$ may be regarded as a lower-order term and treated by standard parabolic theory. The advantage of the Gaussian formulation is of a different nature. Writing
$$
\Delta u-x\cdot Du=\operatorname{div}_{\gamma_d}Du
$$
incorporates the drift into the principal part of the operator and reveals its symmetric and energy-driven structure in $L^2(\Omega,\gamma_d)$. In particular, the Ornstein--Uhlenbeck evolution becomes the gradient flow of the Gaussian Dirichlet energy
$$
\mathcal E(u)=\frac12\int_\Omega |Du|^2\,d\gamma_d.
$$
This point of view is especially useful for our purposes, since it places the equation directly within the variational framework and allows us to exploit the rich toolbox of the calculus of variations, including convexity arguments, variational inequalities, comparison principles, time regularization, and the corresponding methods for nonlinear and obstacle problems developed in the subsequent sections.
\end{remark}


\end{proof}

\section{The Ornstein--Uhlenbeck obstacle problem}\label{sec:obstacle}

Building on Section~\ref{sec:ou-operator}, we here discuss existence and uniqueness of solutions for the Ornstein-Uhlenbeck obstacle problem. Namely, we are able to prove Theorem \ref{thm:intro-obstacle}. This result is obtained through Theorems \ref{thm:obstacle-existence} and \ref{thm:obstacle-uniqueness}, in the spirit of \cite{BDS2017}.

\subsection{Functional framework}\label{subsec:obstacle-framework}

For an obstacle $\psi:\Omega_T\to[-\infty,+\infty)$ and $g$ satisfying
\eqref{eq:ou-data-assumption}, $0<\tau\leq T$, set
$$X_g(\Omega_\tau):=\{w\in L^2(0,\tau;H^1(\Omega,\gamma_d)):w-g\in L^2(0,\tau;V)\},$$
 with $V$ as in \eqref{eq:hilbert-triple} and
\begin{equation}\label{eq:obstacle-admissible-class}
 K_{\psi,g}(\Omega_\tau):=\bigl\{w\in X_g(\Omega_\tau):w\geq\psi\text{ a.e. in }\Omega_\tau\bigr\},
\end{equation}
writing $X_g:=X_g(\Omega_T)$, $K_{\psi,g}:=K_{\psi,g}(\Omega_T)$. We assume
$g\geq\psi$ a.e., so $g\in K_{\psi,g}(\Omega_\tau)$ for every $\tau$.

\begin{definition}[Solution of the obstacle problem]\label{def:obstacle-solution}
$u\in K_{\psi,g}(\Omega_T)\cap C([0,T];H)$, $u(\cdot,0)=g(\cdot,0)$, solves
the obstacle problem if, for every $\tau\in(0,T]$ and every
$v\in K_{\psi,g}(\Omega_\tau)$ with $\partial_tv\in L^2(0,\tau;V^*)$,
\begin{align}\nonumber
 \int_0^\tau \pair{\partial_t v}{v-u}_V\,dt
 +\int_0^\tau\int_\Omega\Bigl[\frac12|Dv|^2-\frac12|Du|^2\Bigr]\,d\gamma_d\,dt\\ \label{eq:obstacle-variational-inequality}
 \geq\frac12\norm{(v-u)(\cdot,\tau)}_H^2 -\frac12\norm{v(\cdot,0)-g(\cdot,0)}_H^2.
\end{align}
\end{definition}

\noindent Note that for smooth data this is the complementarity system
$u\geq\psi$, $\partial_tu-\divg Du\geq0$, $(\partial_tu-\divg Du)(u-\psi)=0$.\\

We aim to prove existence and uniqueness under the hypothesis
\begin{align}\nonumber
 \psi\in L^2(\Omega_T,\gamma_d)\cap L^2(0,T;H^1(\Omega,\gamma_d)),\\ \nonumber
 g\in L^2(0,T;H^1(\Omega,\gamma_d)),\ \partial_tg\in L^2(\Omega_T,\gamma_d),\\ \label{eq:obstacle-sharp-hypothesis}
 g\geq\psi\text{ a.e.}, \quad g_o:=g(\cdot,0)\in H
\end{align}
in the spirit of Bögelein--Duzaar--Scheven \cite{BDS2017}, since this is the Gaussian-divergence counterpart, for $f(\xi)=\frac12|\xi|^2$. Since
$\Omega$ is bounded Lipschitz, the norm equivalence \eqref{eq:norm-equivalence}
transports every estimate of
\cite{BDS2017} into the Gaussian setting at the cost of fixed constants, and
we cite \cite{BDS2017} directly wherever the argument transcribes verbatim
under $dx\mapsto d\gamma_d$, $f\mapsto\frac12|\xi|^2$, giving full detail
only when the argument differs. 

\subsection{Regular data}\label{subsec:obstacle-regular}

Assume
\begin{equation}\label{eq:obstacle-regular-data}
 g\in L^2(0,T;H^1(\Omega,\gamma_d)),\
 \partial_tg\in L^2(\Omega_T,\gamma_d)\cap L^2(0,T;H^1(\Omega,\gamma_d)),\
 g_o\in H^1(\Omega,\gamma_d),
\end{equation}
\begin{equation}\label{eq:obstacle-regular-obstacle}
 \psi\in g+L^2(0,T;H^1_0(\Omega,\gamma_d)),\
 \partial_t\psi\in L^2(\Omega_T,\gamma_d)\cap L^2(0,T;H^1(\Omega,\gamma_d)),\
 \psi_o\in g_o+H^1_0(\Omega,\gamma_d),
\end{equation}
$g\geq \psi $ a.e.\ . As in \cite[Rmk.~4.2]{BDS2017} this gives
$g(t),\psi(t)\in H^1(\Omega,\gamma_d)$ for every $t$. We introduce

\begin{definition}[Strong solution]\label{def:strong-obs}
We say that
$u\in C([0,T];H)\cap X_g$, $u\geq\psi$, $u(\cdot,0)=g_o$,  is a \emph{strong}
solution if \eqref{eq:obstacle-variational-inequality} holds with $\tau=T$
against every $v\in X_g$ with $\partial_tv\in L^2(\Omega_T,\gamma_d)$,
$v(\cdot,0)\in H$, $v\geq\psi$.
\end{definition}

We use throughout the Landes mollification $$[v]_h^{v_o}(t):=e^{-t/h}v_o
+\frac1h\int_0^te^{(s-t)/h}v(s)\,ds.$$
In particular, according to \cite[Lemmas~2.1--2.2]{BDS2017}:\\ 
\noindent(i) $[v]_h^{v_o}\in C([0,\tau];H)\cap L^2(0,\tau;Z)$, $[v]_h^{v_o}(0)=v_o$,
$\partial_t[v]_h^{v_o}=-\frac1h([v]_h^{v_o}-v)$, $[v]_h^{v_o}\to v$ if
$v_o\to v(0)$ in $H$. \\
\noindent(ii) If $v_o=v(0)$ and $\partial_tv\in L^2(0,\tau;H)$, then $\partial_t[v]_h^{v(0)}$
is the convolution of $\partial_tv$ with the kernel $r\mapsto\frac1h e^{-r/h}\mathbf1_{r>0}$,
whose $L^1(\mathbb R)$-norm equals $1$ for every $h$. Young's convolution
inequality then gives
$\|\partial_t[v]_h^{v(0)}\|_{L^2(0,\tau;H)}\leq\|\partial_tv\|_{L^2(0,\tau;H)}$,
with no loss as $h\downarrow0$.\\ 
\noindent(iii) $v\mapsto[v]_h^{v_o}$ is order
preserving (fixed $v_o$) and affine in $v_o$ (fixed $v$).

\begin{theorem}\label{thm:obstacle-regular}
Under assumptions \eqref{eq:obstacle-regular-data}--\eqref{eq:obstacle-regular-obstacle}
there is a unique strong solution $u\in g+L^\infty(0,T;H^1_0(\Omega,\gamma_d))$,
$\partial_tu\in L^2(\Omega_T,\gamma_d)$, and it solves the obstacle problem
in the sense of Definition~\ref{def:obstacle-solution}, with
\begin{align}\nonumber
 \frac12\int_{\Omega_T}|\partial_tu|^2\,d\gamma_d\,dt
 +\frac12\sup_{t}\int_\Omega|Du(t)|^2\,d\gamma_d\\ \label{eq:obstacle-regular-energy}
 \leq e^{T}\Bigl[\frac12\int_\Omega|Dg_o|^2\,d\gamma_d
 +\int_{\Omega_T}\bigl(\tfrac12|\partial_t\psi|^2+|\partial_tD\psi|^2\bigr)d\gamma_d\,dt\Bigr].
\end{align}
\end{theorem}

\begin{proof} We proceed following the steps of \cite[Thm 4.1]{BDS2017}.\\
$\bullet$ \emph{Existence and the energy bound} follow 
with $g_i:=g(ih)$, $\psi_i:=\psi(ih)$, $u_0:=g_o$, let $u_i$
minimize $$\mathcal F_i[w]:=\int_\Omega\frac12|Dw|^2\,d\gamma_d+\frac1{2h}\int_\Omega|w-u_{i-1}|^2\,d\gamma_d$$
over $w\geq\psi_i$, $w-g_i\in H^1_0(\Omega,\gamma_d)$, strictly convex,
coercive by \eqref{eq:weighted-poincare}. We then test against
$u_{i-1}+\psi_i-\psi_{i-1}$. Since
$f(\xi)=\frac12|\xi|^2$ is quadratic, no analogue of \cite{BDS2017}'s
Lipschitz estimate is needed, only Young's inequality on the cross
term $\int_\Omega Du_{i-1}\cdot(D\psi_i-D\psi_{i-1})\,d\gamma_d$.
Summing and
iterating as in \cite[Section~4.1.2]{BDS2017}, we get the discrete bound, and Aubin--Lions
compactness (as in \cite[Section~4.1.3]{BDS2017}) gives a limit $u$ satisfying
\eqref{eq:obstacle-regular-energy} and $u\geq\psi$.\\
$\bullet$ \emph{The variational inequality.} $u^{(h)}$ (piecewise-constant
interpolant) minimizes $$F^{(h)}[w]:=\int_{\Omega_T}(\frac1{2h}|w(t)-u^{(h)}(t-h)|^2+\frac12|Dw|^2)\,d\gamma_d\,dt.$$
Since it is quadratic in $w$, its minimizer is characterized
directly by the vanishing one-sided Gateaux derivative, replacing
\cite{BDS2017}'s convex-combination-and-$s\downarrow0$ device (Section~4.1.4),
needed there only because $F^{(h)}$ is merely convex for general
$p$-growth. Testing against $v+\psi^{(h)}-\psi$ and passing to the limit
$h\downarrow0$ exactly as in \cite[Section~4.1.5]{BDS2017} gives
\eqref{eq:obstacle-variational-inequality} on $[0,T]$, i.e. $u$ is a strong
solution.\\
$\bullet$ \emph{Localization} to Definition~\ref{def:obstacle-solution}, i.e. every $\tau$, not just $\tau=T$, is \cite[Section~3.3.1]{BDS2017}. For
$\tau<T$, $\theta\downarrow0$, the competitor
$\xi_\theta v+(1-\xi_\theta)(\psi+[u-\psi]_h^{g_o-\psi_o})$, which is admissible by
order-preservation, yields, after \cite{BDS2017}'s five-term
$\theta\downarrow0$ decomposition 
and $h\downarrow0$,
every error term vanishing via $[u-\psi]_h^{g_o-\psi_o}\to u-\psi$ in
$C([0,T];H)$ except one integral comparing $D[u-\psi]_h+D\psi$ to
$Du$, where quadratic $f$ again lets Cauchy--Schwarz, i.e.
\begin{multline*}
\Bigl|\int_{\Omega\times(\tau,T)}\Bigl[\frac12|Du+Dw_h|^2-\frac12|Du|^2\Bigr]\,d\gamma_d\,dt\Bigr|\\
\leq\|Dw_h\|_{L^2(\Omega\times(\tau,T),\gamma_d)}
\bigl(\|Du\|_{L^2(\Omega_T,\gamma_d)}+\|Dw_h\|_{L^2(\Omega\times(\tau,T),\gamma_d)}\bigr)\to 0,
\end{multline*}
with $Dw_h\to0$ in $L^2(\Omega\times(\tau,T),\gamma_d)$, replace
\cite{BDS2017}'s Lipschitz estimate.\\
$\bullet$ \emph{Uniqueness}. Let us take $w:=\frac12(u_1+u_2)$, $w_h:=g+R_h(w-g)$, according to Lemma~\ref{lem:time-regularization}, with  zero initial value since $w(\cdot,0)=g_o$. 
Adding the two inequalities tested against $w_h$ and letting $h\downarrow0$, as in Proposition~\ref{prop:var-uniqueness-ou}, using
\eqref{eq:bdm-affine-one-sided} and convexity of $\frac12|\xi|^2$, we infer
$u_1=u_2$.
\end{proof}

We also note that, exactly as \cite[Lemma~2.4]{BDS2017} in the special case
$f=\frac12|\xi|^2$, we have the following.\\ For $u$ solving the obstacle problem
(Definition~\ref{def:obstacle-solution}) and any admissible $v$ with
$\partial_tv\in L^2(\Omega_T,\gamma_d)$,
\begin{align}\nonumber
 \sup_t\|u(t)\|_H^2+\int_{\Omega_T}|Du|^2\,d\gamma_d\,dt\\ \label{eq:refined-energy}
\leq8\Bigl(\int_0^T\|\partial_tv\|_H\,dt\Bigr)^2+4\int_{\Omega_T}|Dv|^2\,d\gamma_d\,dt
 +2\sup_t\|v(t)\|_H^2+4\|v(\cdot,0)-g_o\|_H^2,
\end{align}
via $\int_{\Omega_T}\partial_tv(v-u)\,d\gamma_d\,dt\leq\frac14\sup_t\|(v-u)(t)\|_H^2
+(\int_0^T\|\partial_tv\|_H\,dt)^2$. Observe that this uses the $L^1(0,T;H)$-norm of
$\partial_tv$, essential below.

\subsection{Existence and uniqueness of solutions of the obstacle problem}\label{subsec:obstacle-general-regular}

Dropping the requirement that $\psi,g$ coincide near $\partial\Omega$
(keeping \eqref{eq:obstacle-regular-data}--\eqref{eq:obstacle-regular-obstacle}
otherwise, $\psi\leq g$), we can proceed exactly as in \cite[Section~4.2]{BDS2017}.\\
Let $\zeta_\varepsilon$ be the standard cutoff of
$\Omega_\varepsilon:=\{\operatorname{dist}(\cdot,\partial\Omega)>\varepsilon\}$ and
$\psi_\varepsilon:=\zeta_\varepsilon\psi+(1-\zeta_\varepsilon)g$. Then, Theorem~\ref{thm:obstacle-regular}
gives $u_\varepsilon$ solving the obstacle problem with data $(\psi_\varepsilon,g)$.
Testing by $v=g$ and \eqref{eq:refined-energy}, we obtain a uniform bound for $u_\varepsilon$. We now take advantage of the boundary-strip Poincar\'e estimate
$$\int_{(\Omega\setminus\Omega_\varepsilon)\times(0,T)}|w|^2\,d\gamma_d\,dt
\leq c\varepsilon^2\int_{(\Omega\setminus\Omega_\varepsilon)\times(0,T)}|Dw|^2\,d\gamma_d\,dt$$
for $w\in H^1_0(\Omega,\gamma_d)$ --- the unweighted case is
\cite[Lemma~4.3]{BDS2017}, transported by \eqref{eq:norm-equivalence} ---
which gives $Dv_\varepsilon\to Dv$ for $v_\varepsilon:=\zeta_\varepsilon v+(1-\zeta_\varepsilon)g$, with
$v$ a fixed comparison map. Finally, passing to the limit $\varepsilon\downarrow0$ as
in \cite[Section~4.2]{BDS2017}, it follows

\begin{theorem}\label{thm:obstacle-general-regular}
Let $g$ satisfy \eqref{eq:obstacle-regular-data} and let $\psi\leq g$
satisfy \eqref{eq:obstacle-regular-obstacle} except possibly on
$\partial\Omega\times(0,T)$. Then the obstacle problem
(Definition~\ref{def:obstacle-solution}) has a solution
$u\in L^\infty(0,T;H)\cap(g+L^2(0,T;H^1_0(\Omega,\gamma_d)))$.
\end{theorem}

To reach \eqref{eq:obstacle-sharp-hypothesis}, i.e. $g_o,\psi_o$ in
$H$, \cite[Section~4.3]{BDS2017} mollify in time with a space-mollified initial
value. However, 
our initial value's
gradient necessarily blows up, requiring:

\begin{lemma}\label{lem:mollified-seed}
For $g_o\in H$ (extended by $0$) and a standard mollifier $\varphi$, set
$g_{o,\varepsilon}:=g_o*\varphi_\varepsilon$. Then $g_{o,\varepsilon}\to g_o$ in
$H$ and $\|Dg_{o,\varepsilon}\|_{L^2(\Omega,\gamma_d)}\leq C\varepsilon^{-1}\|g_o\|_H$.
Consequently  $$\|e^{-t/h}Dg_{o,\varepsilon}\|_{L^2(\Omega_T,\gamma_d)}^2\leq
\frac{C^2}2(h\varepsilon^{-2})\|g_o\|_H^2\to0$$ whenever $h\varepsilon^{-2}\to0$,
e.g.\ $\varepsilon=\varepsilon(h):=h^{1/4}$.
\end{lemma}

\begin{proof}
We first observe that $Dg_{o,\varepsilon}=g_o*D\varphi_\varepsilon$. Young's convolution inequality
with $\|D\varphi_\varepsilon\|_{L^1}=\varepsilon^{-1}\|D\varphi\|_{L^1}$ and
\eqref{eq:norm-equivalence} give the gradient bound. The vanishing
estimate is direct integration, $\int_0^Te^{-2t/h}\,dt\leq h/2$.
\end{proof}

\begin{theorem}\label{thm:obstacle-existence}
Under the regularity assumptions \eqref{eq:obstacle-sharp-hypothesis}, the obstacle problem for the
Ornstein--Uhlenbeck operator has a solution
$u\in L^\infty(0,T;H)\cap(g+L^2(0,T;H^1_0(\Omega,\gamma_d)))$, in the sense
of Definition~\ref{def:obstacle-solution}.
\end{theorem}

\begin{proof}
Fix $h_j\downarrow0$, $\varepsilon_j:=h_j^{1/4}$, $g_{o,j}:=g_{o,\varepsilon_j}$
(see Lemma~\ref{lem:mollified-seed}). Set $g_j:=[g]_{h_j}^{g_{o,j}}$, $\psi_j:=[\psi]_{h_j}^{g_{o,j}}$, having with common initial value
$g_{o,j}$.
The shared initial value gives
$$g_j-\psi_j=\frac1{h_j}\int_0^te^{(s-t)/h_j}(g-\psi)(s)\,ds\geq0$$
directly. For fixed $j$, $(\psi_j,g_j)$ satisfies
\eqref{eq:obstacle-regular-data}--\eqref{eq:obstacle-regular-obstacle},
so Theorem~\ref{thm:obstacle-general-regular}
gives $u_j$.\\
\noindent $\bullet$ \emph{Uniform bound}.
We write $g_j=\tilde g_j+e^{-t/h_j}(g_o-g_{o,j})$ with
$\tilde g_j:=[g]_{h_j}^{g_o}$, having matching initial value. Then
\begin{align*}
    \|\partial_t\tilde g_j\|_{L^2(0,T;H)}\leq\|\partial_tg\|_{L^2(\Omega_T,\gamma_d)},\\
    \int_0^T\|\partial_t(e^{-t/h_j}(g_o-g_{o,j}))\|_H\,dt\leq\|g_o-g_{o,j}\|_H, 
    \end{align*}
so $\int_0^T\|\partial_tg_j\|_H\,dt$ is bounded
uniformly. 
Similarly
$$Dg_j=e^{-t/h_j}Dg_{o,j}+\frac1{h_j}\int_0^te^{(s-t)/h_j}Dg(s)\,ds\to Dg$$ in
$L^2(\Omega_T,\gamma_d)$ by Lemma~\ref{lem:mollified-seed} (the residual
vanishes since $h_j\varepsilon_j^{-2}=h_j^{1/2}\to0$), and identically
$D\psi_j\to D\psi$ (same residual, shared initial value). By
\eqref{eq:refined-energy} with $v=g_j$, $u_j$ is bounded uniformly in $j$.\\
$\bullet$ \emph{Passage to the limit}. We have $u_j\rightharpoonup u$ weakly in
$L^2(0,T;H^1(\Omega,\gamma_d))$, weakly-$*$ in $L^\infty(0,T;H)$. Moreover,
$u\geq\psi$ since $\psi_j\to\psi$.\\ 
For $v\in K_{\psi,g}(\Omega_\tau)$, let
$v_j:=[v]_{h_j}^{g_{o,j}}\geq\psi_j$, 
then $Dv_j\to Dv$ as
above. Decomposing $$\partial_tv_j(v_j-u_j)=\partial_t\tilde v_j(v_j-u_j)
+\partial_t(e^{-t/h_j}(v(\cdot,0)-g_{o,j}))(v_j-u_j),$$
with $\tilde v_j:=[v]_{h_j}^{v(\cdot,0)}$, the computation of
\cite[Section~4.3.3, (4.28)--(4.30)]{BDS2017}, adapted to the Gaussian case,
gives
$$\lim_j\int_{\Omega_\tau}\partial_tv_j(v_j-u_j)\,d\gamma_d\,dt
 =\int_{\Omega_\tau}\partial_tv(v-u)\,d\gamma_d\,dt+\frac12\|v(\cdot,0)-g_o\|_H^2.$$
Combined with $Dv_j\to Dv$, this gives
\eqref{eq:obstacle-variational-inequality} for $u,v$.
\end{proof}


\begin{theorem}\label{thm:obstacle-uniqueness}
If, in addition, $\partial_t\psi\in L^2(\Omega_T,\gamma_d)$, the solution of
Theorem~\ref{thm:obstacle-existence} is unique.
\end{theorem}

\begin{proof}
This works analogously to \cite[Lemmas~3.2--3.3]{BDS2017}. Indeed, $v=g$ and \eqref{eq:refined-energy}
give $u_1,u_2\in L^\infty(0,T;H)\cap L^2(0,T;H^1(\Omega,\gamma_d))$. The
comparison map $v_h:=\max\{\psi,[u-g]_h+g\}$ 
upgrades both to $C([0,T];H)$.
We then apply the same construction to $w:=\frac12(u_1+u_2)$ and argue as in
the uniqueness part of Theorem~\ref{thm:obstacle-regular}. The
truncation handled by Lipschitz continuity of $\max\{\cdot,\cdot\}$,
translated via \eqref{eq:norm-equivalence}, gives $u_1=u_2$.
\end{proof}

\begin{proof}[Proof of Theorem~\ref{thm:intro-obstacle}]
This follows by combining Theorem~\ref{thm:obstacle-existence} (existence)
with Theorem~\ref{thm:obstacle-uniqueness} (uniqueness under the additional
hypothesis $\partial_t\psi\in L^2(\Omega_T,\gamma_d)$).
\end{proof}



\section{More general operators in Gaussian divergence form}\label{sec:general-operators}

This section proves Theorem~\ref{thm:intro-general-p}.  We retain the notation
and hypotheses \ref{hyp:p-car}--\ref{hyp:p-derivative} from the introduction,
together with the evolution triple \eqref{eq:p-triple} and the datum assumption
\eqref{eq:p-data-intro}.  In particular, $1<p<\infty$ and
$p\geq2d/(d+2)$.


\subsection{The weighted first variation}

For $u\in W^{1,p}(\Omega,\gamma_d;\mathbb R^N)$, set
\[
  \mathcal F_{\gamma_d}(u)
  :=\int_\Omega f(x,u,Du)\,d\gamma_d.
\]
The upper bound in \eqref{eq:standard-p-growth} ensures that this energy is
finite on $W^{1,p}(\Omega,\gamma_d;\mathbb R^N)$.  If $u$ and $\phi$ are
smooth and $\phi$ has compact support in $\Omega$, then
\begin{equation}\label{eq:first-variation-p}
\begin{aligned}
  \left.\frac{d}{d\varepsilon}\right|_{\varepsilon=0}
     \mathcal F_{\gamma_d}(u+\varepsilon\phi)
  &=\int_\Omega\bigl[D_uf(x,u,Du)\cdot\phi
       +D_\xi f(x,u,Du)\cdot D\phi\bigr]d\gamma_d\\
  &=\int_\Omega\bigl[D_uf(x,u,Du)
       -\divg D_\xi f(x,u,Du)\bigr]\cdot\phi\,d\gamma_d.
\end{aligned}
\end{equation}
Thus the formal negative $L^2(\Omega,\gamma_d;\mathbb R^N)$-gradient flow of
$\mathcal F_{\gamma_d}$ is precisely \eqref{eq:p-main-gaussian-intro}.
Equivalently, in Lebesgue variables,
\[
  \partial_tu-\operatorname{div}_xD_\xi f(x,u,Du)
  +x\cdot D_\xi f(x,u,Du)+D_uf(x,u,Du)=0,
\]
where the contraction with $x$ and the divergence are taken row-wise.

\subsection[Distributional and energy weak formulations]{\textcolor{black}{Distributional and energy weak formulations}}
The derivative bound \eqref{eq:first-derivative-growth} implies, for every
$u\in L^p(0,T;W^{1,p}(\Omega,\gamma_d;\mathbb R^N))$,
\[
  D_\xi f(x,u,Du),\ D_uf(x,u,Du)
  \in L^{p'}(\Omega_T,\gamma_d\otimes dt).
\]
Consequently the spatial part of the equation defines an element of
$L^{p'}(0,T;V_p^*)$.

\begin{proposition}[Distributional solutions are energy weak solutions]
\label{prop:p-distributional-energy}
Suppose that
\[
  u\in L^p\bigl(0,T;W^{1,p}_g(\Omega,\gamma_d;\mathbb R^N)\bigr)
       \cap C([0,T];H),
  \qquad u(\cdot,0)=g(\cdot,0),
\]
and that
\begin{equation}\label{eq:p-distributional}
  \int_{\Omega_T}u\cdot\partial_t\varphi\,d\gamma_d\,dt
  -\int_{\Omega_T}D_\xi f(x,u,Du)\cdot D\varphi\,d\gamma_d\,dt
  -\int_{\Omega_T}D_uf(x,u,Du)\cdot\varphi\,d\gamma_d\,dt=0
\end{equation}
for every $\varphi\in C_c^\infty(\Omega_T;\mathbb R^N)$.  Then
$\partial_tu\in L^{p'}(0,T;V_p^*)$ and $u$ is an energy weak solution in the
sense of Definition~\ref{def:p-weak-intro}.  Conversely, every energy weak
solution satisfies \eqref{eq:p-distributional}.
\end{proposition}

\begin{proof}
Define $F\in L^{p'}(0,T;V_p^*)$ by
\[
  \pair{F(t)}{\phi}_{V_p}
  :=-\int_\Omega D_\xi f(x,u,Du)\cdot D\phi\,d\gamma_d
    -\int_\Omega D_uf(x,u,Du)\cdot\phi\,d\gamma_d.
\]
The distributional identity says that $\partial_tu=F$ in
$\mathcal D'(0,T;V_p^*)$.  Hence $\partial_tu\in L^{p'}(0,T;V_p^*)$, and
\eqref{eq:p-weak} follows by density.  The converse is immediate by testing
\eqref{eq:p-weak} with compactly supported smooth functions and integrating by
parts in time.
\end{proof}

\subsection[Existence and the weak-to-variational implication]{\textcolor{black}{Existence and the weak-to-variational implication}}

\begin{theorem}[Energy weak well-posedness under standard $p$-growth]
\label{thm:p-var-existence}
Suppose that the integrand $ f:\Om\times\R^N\times\R^{N\times d}\longrightarrow [0,\infty)$ satisfies \ref{hyp:p-car}--\ref{hyp:p-derivative} and that $g$ is as in \eqref{eq:p-data-intro}. Then problem
\eqref{eq:p-main-gaussian-intro} has a unique energy weak solution.
\end{theorem}

\begin{proof}
We use the standard monotone-operator theorem in the Hilbert triple
$V_p\hookrightarrow H\hookrightarrow V_p^*$.  Set $w:=u-g$.  For a.e.
$t\in(0,T)$ define $A(t):V_p\to V_p^*$ by
\begin{equation}\label{eq:operator-A}
\begin{aligned}
   \pair{A(t)w}{\phi}_{V_p}
   &:=\int_\Omega D_\xi f(x,w+g(t),Dw+Dg(t))\cdot D\phi\,d\gamma_d \\
   &\quad +\int_\Omega D_u f(x,w+g(t),Dw+Dg(t))\,\phi\,d\gamma_d .
\end{aligned}
\end{equation}
By the derivative growth assumption \eqref{eq:first-derivative-growth}, H\"older's inequality and the continuous embedding
$V_p\hookrightarrow L^p(\Omega,\gamma_d;\mathbb R^N)$,
\begin{equation}\label{eq:operator-growth}
   \|A(t)w\|_{V_p^*}
   \le C\Bigl(1+\|w\|_{V_p}^{p-1}
        +\|g(t)\|_{W^{1,p}(\Omega,\gamma_d)}^{p-1}\Bigr).
\end{equation}
By the Caratheodory assumption, $t\mapsto A_g(t)w$ is measurable for each fixed $w\in V_p$. The map $A(t)$ is hemicontinuous because $f$ is $C^1$ in $(y,\xi)$, and it is
monotone by convexity: if $u_j=w_j+g(t)$, then
\begin{equation}\label{eq:operator-monotone}
\begin{aligned}
   &\pair{A(t)w_1-A(t)w_2}{w_1-w_2}_{V_p} \\
   &=\int_\Omega\Bigl[D_u f(x,u_1,Du_1)-D_u f(x,u_2,Du_2)\Bigr]
        \,(u_1-u_2)d\gamma_d \\
   &\quad+\int_\Omega\Bigl[D_\xi f(x,u_1,Du_1)-D_\xi f(x,u_2,Du_2)\Bigr]
        \cdot(Du_1-Du_2)d\gamma_d\ge0.
\end{aligned}
\end{equation}
It remains to verify coercivity in the variable $w$.  Put $u=w+g(t)$.  By convexity at $(0,0)$,
\[
   D_u f(x,u,Du)\, u+D_\xi f(x,u,Du)\cdot Du
      \ge f(x,u,Du)-f(x,0,0).
\]
Therefore
\begin{align*}
  \pair{A_g(t)w}{w}_{V_p^*,V_p}
  &=\int_\Omega\Bigl[D_uf(x,u,Du)\,(u-g(t))
     +D_\xi f(x,u,Du)\cdot(Du-Dg(t))\Bigr]d\gamma_d \\
  &\ge \int_\Omega f(x,u,Du)\,d\gamma_d
     -\int_\Omega f(x,0,0)\,d\gamma_d \\
  &\quad-
    \int_\Omega\Bigl(|D_uf(x,u,Du)|\,|g(t)|
     +|D_\xi f(x,u,Du)|\,|Dg(t)|\Bigr)d\gamma_d .
\end{align*}
Using \eqref{eq:standard-p-growth}, \eqref{eq:first-derivative-growth}, Young's inequality, and the Poincar\'{e} inequality \eqref{eq:weighted-poincare} for $w\in V_p$, we obtain, for every sufficiently small $\delta>0$,
\[
  \pair{A_g(t)w}{w}_{V_p}
  \ge c\norm{w}_{V_p}^p
  -C\left(1+\norm{g(t)}_{W^{1,p}(\Omega,\gamma_d)}^p+\norm{b}_{L^1(\Omega,\gamma_d)}\right).
\]
Indeed, the terms involving $g(t)$ and $Dg(t)$ are bounded by
\[
  \delta\norm{w}_{V_p}^p
  +C_\delta\left(1+\norm{g(t)}_{W^{1,p}(\Omega,\gamma_d)}^p\right),
\]
and $\delta$ is chosen small enough to absorb the first term into the coercive part.

The shifted problem is
\begin{equation}\label{eq:w-problem-p-fixed}
  \partial_tw+A_g(t)w=-\partial_tg
  \quad\text{in }L^{p'}(0,T;V_p^*),
  \qquad w(\cdot,0)=0.
\end{equation}
The standard Lions theorem for coercive, hemicontinuous, monotone operators gives
\[
   w\in L^p(0,T;V_p)\cap C([0,T];H),
   \qquad \Dt w\in L^{p'}(0,T;V_p^*).
\]
Hence $u=w+g$ is an energy weak solution of \eqref{eq:p-main-gaussian-intro}. Uniqueness follows from the monotonicity in
\eqref{eq:operator-monotone}.

\end{proof}

\begin{proposition}[Energy weak solutions are variational solutions]
\label{prop:p-weak-to-var}
Every energy weak solution of \eqref{eq:p-main-gaussian-intro} is a variational
solution in the sense of Definition~\ref{def:p-var-intro}.
\end{proposition}

\begin{proof}
Let $u$ be an energy weak solution and let $v$ be an admissible comparison map on $(0,\tau)$.  Since $v-u\in L^p(0,\tau;V_p)$, we may test \eqref{eq:p-weak} by $v-u$ on $(0,\tau)$.  This gives
\begin{equation}\label{eq:test-v-minus-u}
\begin{aligned}
   \int_0^\tau\pair{\Dt u}{v-u}_{V_p}\,dt
   &+\int_0^\tau\int_\Omega D_\xi f(x,u,Du)\cdot D(v-u)d\gam\,dt \\
   &+\int_0^\tau\int_\Omega D_u f(x,u,Du)(v-u)d\gam\,dt=0 .
\end{aligned}
\end{equation}
By convexity in the pair $(u,\xi)$ yields
\begin{equation}\label{eq:full-convexity-p}
\begin{aligned}
   f(x,v,Dv)-f(x,u,Du)
   &\ge D_u f(x,u,Du)\,(v-u)  \\
   &\quad +D_\xi f(x,u,Du)\cdot D(v-u).
\end{aligned}
\end{equation}
Combining \eqref{eq:test-v-minus-u} and \eqref{eq:full-convexity-p}, we find
\[
   \int_0^\tau\pair{\Dt u}{v-u}_{V_p}\,dt
   +\int_0^\tau\int_\Omega\bigl[f(x,v,Dv)-f(x,u,Du)\bigr]d\gam\,dt\ge0.
\]
Moreover $v-u\in L^p(0,\tau;V_p)$ and $\Dt(v-u)\in L^{p'}(0,\tau;V_p^*)$.  Hence the Lions--Magenes identity \eqref{eq:lions-magenes} in the triple $V_p\hookrightarrow H\hookrightarrow V_p^*$ gives
\[
   \int_0^\tau\pair{\Dt(v-u)}{v-u}_{V_p}\,dt
   =\frac12\|(v-u)(\cdot,\tau)\|_H^2-\frac12\|v(\cdot,0)-g(\cdot,0)\|_H^2 .
\]
Adding this identity and using
$\Dt v=\Dt u+\Dt(v-u)$ proves \eqref{eq:p-variational-intro}.
\end{proof}

\subsection{Uniqueness in the variational class}

\begin{theorem}[Uniqueness of variational solutions under standard $p$-growth]\label{thm:p-var-uniqueness}
Suppose that the integrand $ f:\Om\times\R^N\times\R^{N\times d}\longrightarrow [0,\infty)$ satisfies \ref{hyp:p-car}--\ref{hyp:p-derivative} and that $g$ is as in \eqref{eq:p-data-intro}. Then the variational solution
of \eqref{eq:p-main-gaussian-intro} is unique. 
\end{theorem}

\begin{proof}
Let $u_1$ and $u_2$ be two variational solutions and fix
$\tau\in(0,T]$.  We argue on the interval $(0,\tau)$ and set
\[
   w:=\frac{u_1+u_2}{2}.
\]
Then $w-g\in L^p(0,\tau;V_p)$, $w\in C([0,\tau];H)$ and
$w(\cdot,0)=g(\cdot,0)$ in $H$.  The only point which prevents us from using
$v=w$ directly in the variational inequality is that we do not know that $\partial_t w$ belongs to $L^{p'}(0,T;V^*)$.  For this reason, for $h \in (0,T]$, we introduce the time-regularised functions
\begin{equation}\label{eq:p-std-regularized-comparison}
  w_h:=g+z_h ,\qquad z_h:=R_hz, 
\end{equation}
where
\[
   z:=w-g,
\]
and where $R_h$ was introduced in \eqref{eq:bdm-zero-seed-regularization}. Since $z(\cdot,0)=0$, the regularization is taken with zero initial datum $z_0$.  By
Lemma~\ref{lem:time-regularization}, $w_h$ is an admissible comparison
map, $w_h(\cdot,0)=g(\cdot,0)$ and
\begin{equation}\label{eq:p-std-regularized-convergence}
   w_h\to w
   \quad\text{in }L^p(0,\tau;W^{1,p}(\Omega,\gamma_d;\mathbb R^N))
   \quad\text{and in }C([0,\tau];H).
\end{equation}
Moreover,
\begin{equation}\label{eq:p-std-zh-derivative}
   \partial_t z_h=-\frac1h(z_h-z)
\end{equation}
in $L^p(0,\tau;V_p)$, and hence in $L^{p'}(0,\tau;V_p^*)$ after the usual
embedding through $H$.

We now apply the variational inequality \eqref{eq:p-variational-intro} first to
$u_1$ and then to $u_2$, with the same comparison map $v=w_h$.  Since
$w_h(\cdot,0)=g(\cdot,0)$, the two initial terms vanish.  Adding the two inequalities,
we obtain
\begin{equation}\label{eq:p-std-added}
\begin{aligned}
   &2\int_0^\tau\pair{\Dt w_h}{w_h-w}_{V_p}\,dt
     +2\int_0^\tau\int_\Omega f(x,w_h,Dw_h)\,d\gamma_d\,dt  \\
   &\quad-\sum_{i=1}^2\int_0^\tau\int_\Omega f(x,u_i,Du_i)\,d\gamma_d\,dt  \\
   &\qquad\ge \frac12\sum_{i=1}^2
          \|(w_h-u_i)(\cdot,\tau)\|_H^2 .
\end{aligned}
\end{equation}
We write
\begin{equation}\label{eq:p-std-Ih-IIh}
   I_h:=\int_0^\tau\int_\Omega f(x,w_h,Dw_h)\,d\gamma_d\,dt,
   \qquad
   II_h:=\int_0^\tau\pair{\Dt w_h}{w_h-w}_{V_p}\,dt .
\end{equation}
We treat these two terms separately.

We first focus on $I_h$.  Set
\[
   \Delta_h^0:=w_h-w,
   \qquad
   \Delta_h^1:=Dw_h-Dw.
\]
By the mean-value formula in the variables $(y,\xi)$,
\begin{equation}\label{eq:p-std-mean-value-energy}
\begin{aligned}
   I_h-I
   &=\int_0^1\int_0^\tau\int_\Omega
       D_u f(x,w+s\Delta_h^0,Dw+s\Delta_h^1)\cdot\Delta_h^0
        \,d\gamma_d\,dt\,ds \\
   &\quad+\int_0^1\int_0^\tau\int_\Omega
       D_\xi f(x,w+s\Delta_h^0,Dw+s\Delta_h^1):\Delta_h^1
        \,d\gamma_d\,dt\,ds,
\end{aligned}
\end{equation}
where
\[
   I:=\int_0^\tau\int_\Omega f(x,w,Dw)\,d\gamma_d\,dt .
\]
Using \eqref{eq:first-derivative-growth}, H\"older's inequality, and the
boundedness of $w_h$ and $w$ in $L^p(0,\tau;W^{1,p}(\Omega,\gamma_d))$,
we find
\begin{equation}\label{eq:p-std-energy-estimate}
\begin{aligned}
   |I_h-I|
   &\le C\int_0^\tau\int_\Omega
        \bigl(1+|w_h|^{p-1}+|w|^{p-1}
              +|Dw_h|^{p-1}+|Dw|^{p-1}\bigr)  \\
   &\hspace{4.3cm}\times
        \bigl(|w_h-w|+|Dw_h-Dw|\bigr)\,d\gamma_d\,dt \\
   &\le C\bigl(1+\|w_h\|_{L^p(0,\tau;W^{1,p})}^{p-1}
               +\|w\|_{L^p(0,\tau;W^{1,p})}^{p-1}\bigr) \\
   &\hspace{4.3cm}\times
        \|w_h-w\|_{L^p(0,\tau;W^{1,p})}.
\end{aligned}
\end{equation}
The last factor tends to zero by \eqref{eq:p-std-regularized-convergence}.  Thus
\begin{equation}\label{eq:p-std-Ih-limit}
   I_h\to I
   \qquad\text{as }h\downarrow0.
\end{equation}

We now treat the second term in \eqref{eq:p-std-Ih-IIh}. Since $w_h-w=z_h-z$, we have
\begin{equation}\label{eq:p-std-IIh-split}
\begin{aligned}
   II_h
   &=\int_0^\tau\pair{\partial_t z_h}{z_h-z}_{V_p}\,dt
     +\int_0^\tau\pair{\partial_t g}{z_h-z}_{V_p}\,dt  \\
   &=:II_h^{(1)}+II_h^{(2)} .
\end{aligned}
\end{equation}
By \eqref{eq:bdm-regularization-derivative},
\begin{equation}\label{eq:p-std-IIh1}
   II_h^{(1)}
   =-\frac1h\int_0^\tau \|z_h-z\|_H^2\,dt\le0.
\end{equation}
Moreover, since $\partial_tg\in L^{p'}(0,T;V_p^*)$ and
$z_h\to z$ in $L^p(0,\tau;V_p)$,
\begin{equation}\label{eq:p-std-IIh2}
   |II_h^{(2)}|
   \le \|\partial_tg\|_{L^{p'}(0,\tau;V_p^*)}
        \|z_h-z\|_{L^p(0,\tau;V_p)}\to0 .
\end{equation}
Consequently,
\begin{equation}\label{eq:p-std-IIh-limit}
   \limsup_{h\downarrow0}II_h\le0 .
\end{equation}

We can now let $h\downarrow0$ in \eqref{eq:p-std-added}.  More precisely,
using \eqref{eq:p-std-Ih-limit}, \eqref{eq:p-std-IIh-limit}, and
$w_h(\tau)\to w(\tau)$ in $H$, we obtain
\begin{equation}\label{eq:p-std-midpoint-ineq-detailed}
\begin{aligned}
   &2\int_0^\tau\int_\Omega f(x,w,Dw)\,d\gamma_d\,dt
     -\sum_{i=1}^2
        \int_0^\tau\int_\Omega f(x,u_i,Du_i)\,d\gamma_d\,dt \\
   &\qquad\ge
     \frac12\sum_{i=1}^2\|(w-u_i)(\cdot,\tau)\|_H^2 .
\end{aligned}
\end{equation}
On the other hand, convexity of $(y,\xi)\mapsto f(x,y,\xi)$ gives, for a.e.
$(x,t)$,
\begin{equation}\label{eq:p-std-pointwise-convexity}
   2f(x,w,Dw)
   =2f\left(x,\frac{u_1+u_2}{2},\frac{Du_1+Du_2}{2}\right)
   \le f(x,u_1,Du_1)+f(x,u_2,Du_2).
\end{equation}
After integration, the left-hand side of
\eqref{eq:p-std-midpoint-ineq-detailed} is therefore nonpositive, whereas
the right-hand side is nonnegative.  Hence both sides must vanish, and in
particular
\begin{equation}\label{eq:p-std-terminal-zero}
   \sum_{i=1}^2\|(w-u_i)(\cdot,\tau)\|_H^2=0 .
\end{equation}
Since $w-u_1=(u_2-u_1)/2$ and $w-u_2=(u_1-u_2)/2$, we get
$u_1(\cdot,\tau)=u_2(\cdot,\tau)$ in $H$.  The time $\tau\in(0,T]$ was arbitrary,
and both solutions have the same initial value $g(\cdot,0)$; therefore
$u_1=u_2$ in $C([0,T];H)$ and hence a.e. in $\Omega_T$
\end{proof}

\subsection{Equivalence of the two formulations}

\begin{theorem}[Equivalence under standard $p$-growth]\label{thm:p-standard-existence}
Suppose that the integrand $ f:\Om\times\R^N\times\R^{N\times d}\longrightarrow [0,\infty)$ satisfies \ref{hyp:p-car}--\ref{hyp:p-derivative} and that $g$ is as in \eqref{eq:p-data-intro}. Then
\eqref{eq:p-main-gaussian-intro} has a unique variational solution.  This solution is the
unique weak solution.  Consequently, the weak and variational formulations
are equivalent.
\end{theorem}

\begin{proof}
Theorem~\ref{thm:p-var-existence} gives a weak solution, and
Proposition~\ref{prop:p-weak-to-var} shows that this weak solution is a variational
solution.  Theorem~\ref{thm:p-var-uniqueness} shows that no other variational
solution exists.  Hence every variational solution coincides with the weak
solution constructed above and is therefore a weak solution as well.  The uniqueness
of the weak solution also follows by testing the difference of two weak
formulations by their difference and using the monotonicity identity
\eqref{eq:operator-monotone}.
\end{proof}
\begin{proof}[Proof of Theorem~\ref{thm:intro-general-p}]
The assertion is exactly Theorem~\ref{thm:p-standard-existence}.
\end{proof}

\subsection{Selected model equations and applications}\label{subsec:standard-p-models}
{\color{black}
We conclude by highlighting three representative classes of equations covered by the
preceding framework. They illustrate nonlinear diffusion in Gaussian space, nonlinear
dependence on the unknown, and anisotropic Ornstein--Uhlenbeck operators.

\begin{enumerate}[label=\textup{(M\arabic*)},leftmargin=2.4em]
\item \textbf{The Gaussian $p$-Laplacian.}
Taking $f(\xi)=|\xi|^p/p$ gives
\begin{equation}\label{eq:model-gaussian-p-lap}
  \partial_tu=\divg\bigl(|Du|^{p-2}Du\bigr).
\end{equation}
This is the natural nonlinear analogue of the Ornstein--Uhlenbeck evolution obtained by
replacing the quadratic Gaussian Dirichlet energy with the $p$-energy. The Gaussian
$p$-Laplacian arises naturally in weighted spectral and variational problems and provides
a basic model for nonlinear diffusion with Gaussian reference measure; see, for instance,
\cite{ColesantiQinSalani2026}. The Ornstein--Uhlenbeck drift is incorporated into the
principal Gaussian-divergence operator.

\item \textbf{Semilinear Ornstein--Uhlenbeck equations.}
For $p=2$, let $W:\mathbb R^N\to[0,\infty)$ be convex and satisfy the standing growth
assumptions. The energy density
\[
  f(u,\xi)=\frac12|\xi|^2+W(u)
\]
gives
\begin{equation}\label{eq:model-vector-reaction}
  \partial_tu=\divg Du-DW(u).
\end{equation}
Allowing the energy density to depend on $u$, and not only on $Du$, makes it possible
to include a nonlinear term depending on the solution itself within the same variational
framework. Semilinear elliptic equations with Gauss measure are studied, for instance,
by D\'iaz, Feo and Posteraro \cite{DiazFeoPosteraro2021}; in the notation used here, a
prototype is
\[
  -\divg Du+c_0|u|^{q-1}u=F.
\]
Terms depending explicitly on the unknown also occur naturally in Kolmogorov equations.
For example, Issoglio and Russo \cite{IssoglioRusso2024} consider
\[
  \partial_t v+\frac12\Delta v+Dv\cdot b=\lambda v+g.
\]
These examples motivate the dependence of the general integrand $f(x,u,Du)$ on the value
variable $u$.

\item \textbf{Anisotropic Ornstein--Uhlenbeck operators.}
In the scalar case, let $A\in\mathbb R^{d\times d}$ be symmetric and positive definite.
The quadratic energy density $f(\xi)=\frac12 A\xi\cdot\xi$ yields
\begin{equation}\label{eq:model-anisotropic-ou}
  \partial_tu=\divg(A Du)
  =\operatorname{tr}(A D^2u)-(Ax)\cdot Du.
\end{equation}
More generally, finite-dimensional Ornstein--Uhlenbeck operators are of the form
\[
  \mathcal A=\operatorname{Tr}(QD^2)+\langle Bx,D\rangle;
\]
see, for instance, \cite{LunardiMetafunePallara2020}. Equation~\eqref{eq:model-anisotropic-ou}
corresponds to the symmetric choice $Q=A$ and $B=-A$. In particular, this example shows
that the variational formulation is not tied to the isotropic Laplacian considered in
Section~\ref{sec:ou-operator}.
\end{enumerate}
}

\section{The kinetic case}\label{sec:kinetic}
We conclude by explaining how the Gaussian-divergence formulation and the
comparison inequality extend to the force-free kinetic Fokker--Planck
equation without an obstacle.  We work with periodic position variables, so
that no lateral kinetic boundary conditions are needed.  The linear Cauchy
theory is already available in \cite{AAMN}; our purpose is to identify the
corresponding variational formulation and its relation to the preceding
sections, not to develop a theory of kinetic obstacle problems.

\subsection{Mixed Gaussian divergence and the velocity energy}

Set
\begin{equation}\label{kin:eq:mixed-measure-short}
 Q_T:=(0,T)\times\Torus^d_x\times\R^d_v,
 \qquad d\mu(x,v):=dx\,d\gamma_d(v),
\end{equation}
where $dx$ is Lebesgue measure on the torus and
$d\gamma_d(v)=(2\pi)^{-d/2}e^{-|v|^2/2}\,dv$.  We use $Q_\tau$ for the
same cylinder with $T$ replaced by $\tau$, and all phase-space integrals
below are over $\Torus^d\times\R^d$.
For a vector field $\mathbf F=(F_x,F_v)$, define
\begin{equation}\label{kin:eq:mixed-divergence-short}
 \operatorname{div}_{x,\gamma_v}\mathbf F
 :=\operatorname{div}_x F_x+\operatorname{div}_v F_v-v\cdot F_v.
\end{equation}
The corresponding integration-by-parts formula is
\begin{equation}\label{kin:eq:mixed-ibp-short}
 \int\varphi\,\operatorname{div}_{x,\gamma_v}\mathbf F\,d\mu
 =-\int\bigl(D_x\varphi\cdot F_x+D_v\varphi\cdot F_v\bigr)\,d\mu
\end{equation}
for smooth periodic fields $\mathbf F$ and smooth test functions
$\varphi$ compactly supported in $v$. In particular,
\begin{equation}\label{kin:eq:mixed-flux-short}
 \operatorname{div}_{x,\gamma_v}(-vu,D_vu)
 =-v\cdot D_xu+\Delta_vu-v\cdot D_vu.
\end{equation}
Consequently, the kinetic equation can be written in either of the forms
\begin{equation}\label{kin:eq:kinetic-pde-short}
 \begin{aligned}
  \partial_tu+v\cdot D_xu&=\Delta_vu-v\cdot D_vu,\\
  \partial_tu&=\operatorname{div}_{x,\gamma_v}(-vu,D_vu).
 \end{aligned}
\end{equation}

Here the Gaussian weight is part of the formulation. If
$F(t,x,v)=\rho_d(v)u(t,x,v)$, direct differentiation gives the equivalent
Lebesgue-density equation
\[
 \partial_tF+v\cdot D_xF=\Delta_vF+\operatorname{div}_v(vF).
\]
Thus $u$ is the density relative to $dx\,d\gamma_d(v)$ when $F$ is
interpreted as a density relative to $dx\,dv$.

Let
\begin{equation}\label{kin:eq:spaces-short}
 H:=L^2(\Torus^d\times\R^d,d\mu),
 \qquad
 V:=L^2\bigl(\Torus^d;H^1(\R^d,\gamma_d)\bigr),
\end{equation}
with $\|u\|_V^2:=\|u\|_H^2+\|D_vu\|_{L^2(\mu)}^2$.  These spaces form
the Hilbert triple $V\hookrightarrow H\equiv H^*\hookrightarrow V^*$, where
\begin{equation}\label{kin:eq:dual-identification-short}
 V^*\simeq L^2\bigl(\Torus^d;H^{-1}(\R^d,\gamma_d)\bigr),
 \qquad H^{-1}(\R^d,\gamma_d):=\bigl(H^1(\R^d,\gamma_d)\bigr)^*.
\end{equation}
The velocity Dirichlet energy and its first variation are
\begin{equation}\label{kin:eq:A-energy-short}
 \mathcal E(u):=\frac12\int|D_vu|^2\,d\mu,
 \qquad
 \pair{Au}{\varphi}_{V^*,V}:=\int D_vu\cdot D_v\varphi\,d\mu.
\end{equation}
Thus $A=-\operatorname{div}_{\gamma_v}D_v$, and, with
$Y:=\partial_t+v\cdot D_x$, the equation becomes $Yu+Au=0$.

The mixed measure is important here.  Velocity diffusion is symmetric with
respect to $\gamma_d(v)$, whereas periodicity in $x$ gives
\[
 \int (v\cdot D_xu)\,w\,d\mu
 =-\int u\,(v\cdot D_xw)\,d\mu
\]
for smooth periodic functions compactly supported in $v$.
The full kinetic operator therefore consists of a
skew transport operator and the symmetric velocity diffusion.  In
particular, the kinetic evolution is not the $H$-gradient flow of
$\mathcal E$ alone: that gradient flow would omit $v\cdot D_xu$.

\begin{remark}[Mixed and full Gaussian measures]
The distinction can also be seen from the calculation in the full space
$\R^d_x\times\R^d_v$.  For
$d\Gamma(x,v):=d\gamma_d(x)\,d\gamma_d(v)$ one has
\[
 \operatorname{div}_\Gamma(vu,0)
 =v\cdot D_xu-(x\cdot v)u.
\]
Adding the velocity flux $-xu$ cancels the last term, but also introduces
$-x\cdot D_vu$.  The full Gaussian measure consequently gives the identity
\begin{equation}\label{eq:kinetic-gaussian-divergence}
 \operatorname{div}_\Gamma(vu,-xu-D_vu)
 =v\cdot D_xu-x\cdot D_vu-\Delta_vu+v\cdot D_vu,
\end{equation}
which corresponds to the confined equation
\begin{equation}\label{eq:kinetic-ou}
 \partial_tu+v\cdot D_xu-x\cdot D_vu=\Delta_vu-v\cdot D_vu.
\end{equation}
For the force-free periodic model \eqref{kin:eq:kinetic-pde-short}, we use
$dx\,d\gamma_d(v)$ throughout.
\end{remark}

\subsection{Kinetic weak solutions and the Green identity}

At the natural energy regularity, the controlled derivative is $Yu$;
$\partial_tu$ and $v\cdot D_xu$ need not belong separately to
$L^2(0,T;V^*)$.  Accordingly, set
\begin{equation}\label{kin:eq:hkin-short}
 H^1_{\mathrm{kin}}(Q_\tau)
 :=\bigl\{u\in L^2(0,\tau;V):Yu\in L^2(0,\tau;V^*)\bigr\},
\end{equation}
with norm $\|u\|_{L^2(0,\tau;V)}+\|Yu\|_{L^2(0,\tau;V^*)}$.
This is the periodic Gaussian kinetic Sobolev framework of
\cite[Section~6]{AAMN}.

\begin{proposition}[Periodic kinetic traces and Green identity]
\label{kin:prop:periodic-green}
Every $u\in H^1_{\mathrm{kin}}(Q_\tau)$ has a unique representative in
$C([0,\tau];H)$. For $u,w\in H^1_{\mathrm{kin}}(Q_\tau)$ and
$0\leq s\leq t\leq\tau$,
\begin{equation}\label{kin:eq:trace-green-short}
 (u(t),w(t))_H-(u(s),w(s))_H
 =\int_s^t\bigl[\pair{Yu}{w}_{V}+\pair{Yw}{u}_{V}\bigr]\,dr.
\end{equation}
In particular,
\begin{equation}\label{kin:eq:trace-bound-short}
 \sup_{0\leq t\leq\tau}\|u(t)\|_H^2
 \leq\tau^{-1}\|u\|_{L^2(0,\tau;H)}^2+
    2\|Yu\|_{L^2(0,\tau;V^*)}\|u\|_{L^2(0,\tau;V)}.
\end{equation}
\end{proposition}
\begin{proof}
We give a periodic proof that uses only spatial Fourier projections.
First, Gaussian integration by parts for a smooth compactly supported
function $q(v)$ gives, for each $i$,
\[
 \int v_i^2 q^2\,d\gamma_d
 =\int q^2\,d\gamma_d+2\int v_iq\,\partial_iq\,d\gamma_d
 \leq\|q\|_{L^2(\gamma_d)}^2+
       \tfrac12\|v_iq\|_{L^2(\gamma_d)}^2+
       2\|\partial_iq\|_{L^2(\gamma_d)}^2.
\]
Summing and passing to the closure gives
\begin{equation}\label{kin:eq:gaussian-moment-bound}
 \|vq\|_{L^2(\gamma_d)}^2
 \leq2d\|q\|_{L^2(\gamma_d)}^2+4\|D_vq\|_{L^2(\gamma_d)}^2
 \quad (q\in H^1(\mathbb R^d,\gamma_d)).
\end{equation}
The density needed here follows by truncation in $v$ and ordinary
mollification on bounded sets.
Let $P_n$ be the Fourier projection in $x$ onto modes with
$|k|\leq n$, taking symmetric modes so that real-valued functions remain
real-valued. These projections converge strongly to the identity on
$V$, $H$, and $V^*$ and commute with $Y$ and $D_v$. For $u_n=P_nu$,
\eqref{kin:eq:gaussian-moment-bound} and the finite number of spatial
modes imply $v\cdot D_xu_n\in L^2(0,\tau;H)$. Hence
\[
 \partial_tu_n=P_nYu-v\cdot D_xu_n\in L^2(0,\tau;V^*).
\]
The ordinary Hilbert-triple energy identity applies to $u_n$.
Periodicity gives $\int(v\cdot D_xu_n)u_n\,d\mu=0$, so
\[
 \|u_n(t)\|_H^2-\|u_n(s)\|_H^2
 =2\int_s^t\pair{Yu_n}{u_n}_{V^*,V}\,dr.
\]
Choosing one time at which the squared norm is at most its time average
and using Cauchy--Schwarz proves \eqref{kin:eq:trace-bound-short} for
$u_n$. The same estimate for $u_n-u_m$, together with convergence in
$L^2(0,\tau;V)$ and of the material derivatives in $L^2(0,\tau;V^*)$,
shows that $(u_n)$ is Cauchy in $C([0,\tau];H)$. Its limit is the
continuous representative of $u$. Passing to the limit gives the
energy identity, and polarization gives \eqref{kin:eq:trace-green-short}.
\end{proof}

Given $u_0\in H$, an energy weak solution of
\eqref{kin:eq:kinetic-pde-short} is a function
$u\in H^1_{\mathrm{kin}}(Q_T)$ with $u(0)=u_0$ and
\begin{equation}\label{kin:eq:unconstrained-weak}
 \int_0^T\pair{Yu}{\varphi}_{V^*,V}\,dt
 +\int_0^T\int D_vu\cdot D_v\varphi\,d\mu\,dt=0
 \qquad\text{for every }\varphi\in L^2(0,T;V).
\end{equation}
For smooth test functions, periodic in $x$ and compactly supported in time
and velocity, the equivalent distributional identity is
\begin{equation}\label{kin:eq:weighted-distribution-pairing}
 -\int_{Q_T}uY\varphi\,d\mu\,dt
 +\int_{Q_T}D_vu\cdot D_v\varphi\,d\mu\,dt=0.
\end{equation}
Conversely, for $u\in L^2(0,T;V)$ this identity gives $Yu=-Au$ in
$L^2(0,T;V^*)$, and hence the kinetic time trace and the energy weak
formulation.  The initial value is prescribed through that trace.

The periodic linear Cauchy theorem
\cite[Proposition~6.10]{AAMN}, with no force and unit velocity diffusion,
gives a unique such energy weak solution for every $u_0\in H$. The cited Cauchy theorem and the periodic Green identity proved
above are the analytic inputs for the comparison argument. The kinetic
Cauchy theorem is not a consequence of the parabolic theorem in
Section~\ref{sec:general-operators}.

\subsection[The variational formulation]{The variational formulation}

The comparison map is denoted by $w$, since $v$ is the velocity variable.
The replacements in the parabolic formulation are
\[
 \partial_t\ \longrightarrow\ Y,\qquad
 D\ \longrightarrow\ D_v,\qquad
 d\gamma_d(x)\ \longrightarrow\ dx\,d\gamma_d(v).
\]

\begin{definition}[Variational solution of the kinetic equation]
\label{kin:def:unconstrained-variational}
Let $u_0\in H$.  A function
\[
 u\in L^2(0,T;V)\cap C([0,T];H),\qquad u(0)=u_0,
\]
is a variational solution of \eqref{kin:eq:kinetic-pde-short} if, for every
$\tau\in(0,T]$ and every $w\in H^1_{\mathrm{kin}}(Q_\tau)$,
\begin{equation}\label{kin:eq:unconstrained-variational}
 \begin{aligned}
 &\int_0^\tau\pair{Yw}{w-u}_{V}\,dt
   +\frac12\int_0^\tau\int
                    \bigl(|D_vw|^2-|D_vu|^2\bigr)\,d\mu\,dt\\
 &\qquad\geq
   \frac12\|(w-u)(\tau)\|_H^2
   -\frac12\|w(0)-u_0\|_H^2.
 \end{aligned}
\end{equation}
\end{definition}

No estimate for $Yu$ is assumed in this definition.  Only the comparison
map has a material derivative in the dual space.  Nor is $w(0)=u_0$
required: the initial discrepancy is included explicitly on the
right-hand side.

\begin{proposition}[Equivalence for the linear kinetic equation]
\label{kin:prop:unconstrained-equivalence}
For every $u_0\in H$, the energy weak solution of
\eqref{kin:eq:kinetic-pde-short} is the unique variational solution in the
sense of Definition~\ref{kin:def:unconstrained-variational}.
In particular, every variational solution belongs to
$H^1_{\mathrm{kin}}(Q_T)$ and satisfies
\eqref{kin:eq:unconstrained-weak}.
\end{proposition}

\begin{proof}
Let $U\in H^1_{\mathrm{kin}}(Q_T)$ be the energy weak solution provided by
the linear Cauchy theorem.  Given $w\in H^1_{\mathrm{kin}}(Q_\tau)$, test
the equation for $U$ by $w-U$.  The Green identity
\eqref{kin:eq:trace-green-short} and the quadratic energy identity give
\begin{equation}\label{kin:eq:unconstrained-remainder}
 \begin{aligned}
 &\int_0^\tau\pair{Yw}{w-U}_{V}\,dt
   +\frac12\int_0^\tau\int
                   \bigl(|D_vw|^2-|D_vU|^2\bigr)\,d\mu\,dt\\
 &=\frac12\|(w-U)(\tau)\|_H^2
   -\frac12\|w(0)-u_0\|_H^2
   +\frac12\int_0^\tau\int|D_v(w-U)|^2\,d\mu\,dt.
 \end{aligned}
\end{equation}
The last term is nonnegative, so $U$ is a variational solution.

Conversely, let $u$ be any variational solution and use $w=U$ as a
comparison map in \eqref{kin:eq:unconstrained-variational}.  The initial
term vanishes.  Since $YU=-AU$, the left-hand side is exactly
\[
 \begin{aligned}
 &-\int_0^\tau\int D_vU\cdot D_v(U-u)\,d\mu\,dt
 +\frac12\int_0^\tau\int(|D_vU|^2-|D_vu|^2)\,d\mu\,dt\\
 &\qquad=-\frac12\int_0^\tau\int|D_v(U-u)|^2\,d\mu\,dt.
 \end{aligned}
\]
Therefore
\[
 \frac12\|(U-u)(\tau)\|_H^2
 +\frac12\int_0^\tau\int|D_v(U-u)|^2\,d\mu\,dt\leq0.
\]
It follows that $u(\tau)=U(\tau)$ in $H$ for every $\tau\in(0,T]$.
Thus $u=U$, which proves uniqueness and the asserted material-derivative
regularity without regularizing $u$ in the kinetic variables.
\end{proof}

\begin{proof}[Proof of Proposition~\ref{prop:intro-kinetic}]
This is exactly Proposition~\ref{kin:prop:unconstrained-equivalence}.
\end{proof}

\end{document}